\documentclass[11pt,a4paper]{article}

\usepackage[utf8]{inputenc}
\PassOptionsToPackage{hyphens}{url}
\usepackage{amsmath,amssymb,amsthm}
\usepackage{mathtools}
\usepackage{graphicx}
\usepackage{booktabs}
\usepackage[hypertexnames=false,colorlinks=true,linkcolor=blue,citecolor=blue,urlcolor=blue]{hyperref}
\usepackage{geometry}
\usepackage{float}

\makeatletter
\def\Cref@getprefix#1:#2\relax{#1}
\def\Cref@prefix@sec{sec}
\def\Cref@prefix@fig{fig}
\def\Cref@prefix@tab{tab}
\def\Cref@prefix@eq{eq}
\def\Cref@prefix@thm{thm}
\def\Cref@prefix@prop{prop}
\def\Cref@prefix@lem{lem}
\def\Cref@prefix@cor{cor}
\def\Cref@prefix@conj{conj}
\def\Cref@prefix@def{def}
\def\Cref@prefix@rem{rem}
\def\Cref@prefix@ex{ex}

\def\Cref@type@cap#1{%
  \edef\Cref@tmp{#1}%
  \ifx\Cref@tmp\Cref@prefix@sec Section%
  \else\ifx\Cref@tmp\Cref@prefix@fig Figure%
  \else\ifx\Cref@tmp\Cref@prefix@tab Table%
  \else\ifx\Cref@tmp\Cref@prefix@eq Equation%
  \else\ifx\Cref@tmp\Cref@prefix@thm Theorem%
  \else\ifx\Cref@tmp\Cref@prefix@prop Proposition%
  \else\ifx\Cref@tmp\Cref@prefix@lem Lemma%
  \else\ifx\Cref@tmp\Cref@prefix@cor Corollary%
  \else\ifx\Cref@tmp\Cref@prefix@conj Conjecture%
  \else\ifx\Cref@tmp\Cref@prefix@def Definition%
  \else\ifx\Cref@tmp\Cref@prefix@rem Remark%
  \else\ifx\Cref@tmp\Cref@prefix@ex Example%
  \else Ref.%
  \fi\fi\fi\fi\fi\fi\fi\fi\fi\fi\fi\fi
}

\def\Cref@type@low#1{%
  \edef\Cref@tmp{#1}%
  \ifx\Cref@tmp\Cref@prefix@sec section%
  \else\ifx\Cref@tmp\Cref@prefix@fig figure%
  \else\ifx\Cref@tmp\Cref@prefix@tab table%
  \else\ifx\Cref@tmp\Cref@prefix@eq equation%
  \else\ifx\Cref@tmp\Cref@prefix@thm theorem%
  \else\ifx\Cref@tmp\Cref@prefix@prop proposition%
  \else\ifx\Cref@tmp\Cref@prefix@lem lemma%
  \else\ifx\Cref@tmp\Cref@prefix@cor corollary%
  \else\ifx\Cref@tmp\Cref@prefix@conj conjecture%
  \else\ifx\Cref@tmp\Cref@prefix@def definition%
  \else\ifx\Cref@tmp\Cref@prefix@rem remark%
  \else\ifx\Cref@tmp\Cref@prefix@ex example%
  \else ref.%
  \fi\fi\fi\fi\fi\fi\fi\fi\fi\fi\fi\fi
}

\DeclareRobustCommand{\Cref}[1]{%
  \begingroup
  \edef\Cref@label{#1}%
  \edef\Cref@prefix{\expandafter\Cref@getprefix\Cref@label:\relax}%
  \hyperref[#1]{\Cref@type@cap{\Cref@prefix}~\ref*{#1}}%
  \endgroup
}
\DeclareRobustCommand{\cref}[1]{%
  \begingroup
  \edef\Cref@label{#1}%
  \edef\Cref@prefix{\expandafter\Cref@getprefix\Cref@label:\relax}%
  \hyperref[#1]{\Cref@type@low{\Cref@prefix}~\ref*{#1}}%
  \endgroup
}
\makeatother

\makeatletter
\long\def\@makecaption#1#2{\vskip\abovecaptionskip\small\sbox\@tempboxa{#1: #2}\ifdim\wd\@tempboxa>\hsize #1: #2\par\else\global\@minipagefalse\hb@xt@\hsize{\hfil\box\@tempboxa\hfil}\fi\vskip\belowcaptionskip}
\makeatother

\theoremstyle{plain}
\newtheorem{theorem}{Theorem}[section]
\newtheorem{lemma}[theorem]{Lemma}
\newtheorem{proposition}[theorem]{Proposition}
\newtheorem{corollary}[theorem]{Corollary}

\theoremstyle{definition}
\newtheorem{definition}[theorem]{Definition}
\newtheorem{example}[theorem]{Example}

\theoremstyle{remark}
\newtheorem{remark}[theorem]{Remark}

\newcommand{\R}{\mathbb{R}}
\newcommand{\Z}{\mathbb{Z}}

\title{\textbf{Magic Gems: Convex Hulls of Lifted Magic Squares}}

\author{Kyle Elliott Mathewson\\
\textit{Faculty of Science, University of Alberta}\\
\texttt{kyle.mathewson@ualberta.ca}}

\date{}

\begin{document}

\maketitle

\begin{abstract}
Given an arrangement of $\{1,\dots,n^2\}$ in an $n\times n$ grid, we lift each cell to $\R^3$ by centring its grid coordinates $(x,y)$ at the origin and taking its centred entry $z = s_{ij} - (n^2+1)/2$ as a vertical height. The convex hull of these $n^2$ points is the \emph{gem} of the arrangement, whose lower and upper boundaries are the regular subdivisions of the grid induced by $\pm z$. While random arrangements and generic linear magic squares have gems whose vertex counts grow with $n$, our main theorem shows that the classical Siamese (de la Loub\`ere) magic square of every odd order $n\ge5$ has $f$-vector $(12,26,16)$, isometry group $\{\pm\mathrm{id}\}$, and only the twelve vertices forced by the boundary of the grid. The proof shears the lift onto five horizontal terraces and reduces the hull calculation to a finite table of polynomials in $m=(n-1)/2$; the same method gives $f$-vector $(12,24,14)$ for the base-$n$ digit-exchanged square. Along any family of linear squares with fixed steps the $f$-vector is eventually periodic in $n$, whereas over all linear magic squares of prime order $n$ the maximum vertex count grows as $\Theta(\ln n)$. Whenever the heights have vanishing first moments, a short principal-axes argument shows that the point set has a non-trivial isometry if and only if the arrangement is self-complementary up to dihedral symmetry. An exact census of the $880$ magic squares of order four yields $616$ congruence classes and $352$ combinatorial types of gems, and shows that while vertex and facet counts are nearly insensitive to the magic line-sum conditions, the volume and solid centroid shift systematically.
\medskip


\end{abstract}

\section{Introduction}
\label{sec:introduction}

Magic squares are among the oldest objects of recreational mathematics: the $3\times3$ Lo Shu belongs to Chinese tradition, a $4\times4$ square is engraved in Albrecht D\"urer's \emph{Melencolia~I} (1514) alongside a balance scale and a $12$-vertex convex polyhedron (\emph{D\"urer's solid}, a truncated rhombohedron), and constructions for every order $n\ge3$ have been known for centuries \cite{andrews1960magic,pickover2002zen}.
The modern literature treats them from several directions: enumeration and the geometry of the polytopes whose lattice points are the squares \cite{stanley1973magic,ehrhart1973magic,ahmed2003polyhedral,beck2006insideout,beck2010enumeration}, spectral and matrix properties \cite{loly2009spectra,nordgren2012properties,nordgren2014spectra,cameron2013signatura}, and, in work of Loly and collaborators, physical invariants such as the moment of inertia and the multipole moments \cite{loly2004invariance,rogers2004inertia,rogers2005multipole}.
Magic squares have also been read as height fields: Knecht's water-retention problem asks how much water a magic square of order $n$ can hold when its entries are the heights of the cells \cite{oeis_A201126}, and led to the study of water retention on random surfaces \cite{knecht2012retention}.

\begin{definition}[Magic square]
\label{def:magic_square}
An \emph{arrangement} of order $n$ is a bijection $S=(s_{ij})$ from the cells of the $n\times n$ grid onto $\{1,\dots,n^2\}$ (a permutation of $1,\dots,n^2$ written into the grid).
It is a \emph{magic square} if every row, every column and both main diagonals sum to the magic constant $\mu_n=n(n^2+1)/2$, and \emph{semi-magic} if only the rows and columns do.
\end{definition}

\paragraph{The construction.}
Centre the grid at the origin and attach to the cell $(i,j)$ of an arrangement the point
\[
\mathbf p_{ij}=(x_{ij},y_{ij},z_{ij}),
\]
whose horizontal part $(x_{ij},y_{ij})$ is the position of the cell and whose height $z_{ij}=s_{ij}-\frac{n^2+1}{2}$ is the entry minus the mean entry.
The \emph{gem} $\mathcal G(S)$ is the convex hull of these $n^2$ points (\Cref{sec:gem}).
It is defined for every arrangement: the points are the lift of a fixed grid by a permutation of fixed heights, and the lower and upper faces of the gem are the regular subdivisions $\mathcal S(z)$ and $\mathcal S(-z)$ of the grid, glued along its boundary (\Cref{prop:subdivisions}).
Magic squares are the motivating case, and we call the gem of a magic square a \emph{Magic Gem}: read as vectors from the origin, its points have vertical components that cancel along every row, column and both diagonals (\Cref{lem:lines}; Figures~\ref{fig:vectors} and~\ref{fig:lines}).
Uniting the three emblems of \emph{Melencolia~I}, the equal row and column sums place the lifted heights in mechanical balance about the origin, and the gem of D\"urer's square is itself a centrally symmetric $12$-vertex polyhedron with $f$-vector $(12,24,14)$ and $4$ interior points (\Cref{fig:vectors}).

\begin{figure}[htbp]
\centering
\includegraphics[width=0.98\textwidth]{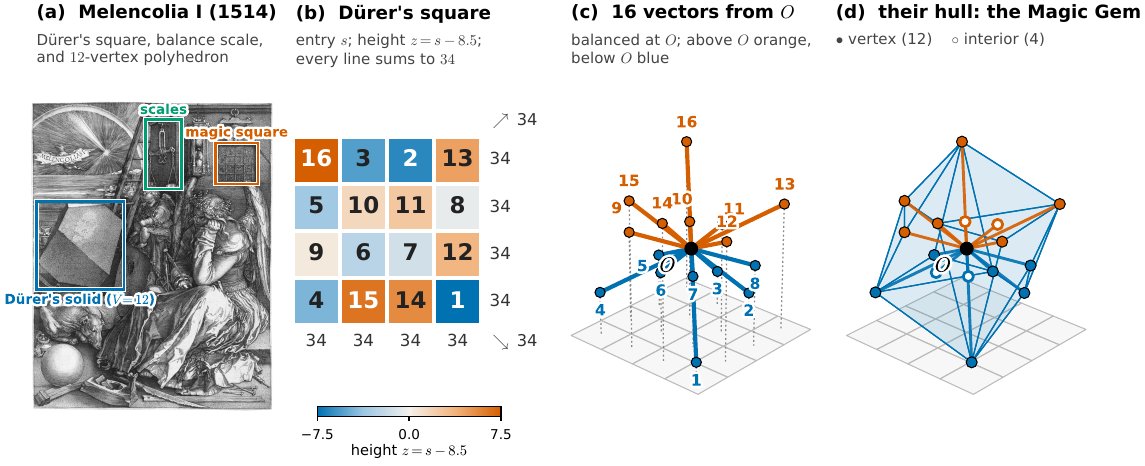}
\caption{From \emph{Melencolia~I} to the Magic Gem. (a)~Albrecht D\"urer's engraving \emph{Melencolia~I} (1514, public domain), highlighting the $4\times4$ magic square, the balance scale, and D\"urer's $12$-vertex truncated rhombohedron ($f=(12,18,8)$). (b)~D\"urer's square with entries $s_{ij}$ on a heat map of the heights $z_{ij}=s_{ij}-8.5$ (orange above $0$, blue below); every line sums to $34$. (c)~The $16$ vectors $\mathbf p_{ij}$ from the origin $O$, labelled by the entries, with drop lines to the floor grid. (d)~Their convex hull, the Magic Gem ($f=(12,24,14)$): $12$ tips are vertices (filled) and $4$ lie in the interior (open). The colour bar is shared by (b) and (c).}
\label{fig:vectors}
\end{figure}

\paragraph{Results.}
The main result is a theorem about one classical family.
\emph{For the Siamese (de la Loub\`ere) square of every odd order $n\ge5$ the vertices of the gem are exactly the twelve points that the boundary of the grid forces (the corners and, on each side, the cells of the smallest and the largest entry); the $f$-vector is $(12,26,16)$ and the isometry group is $\{\pm\mathrm{id}\}$} (\Cref{thm:siamese}).
So only $12$ of the $n^2$ points are vertices, whereas the gems of random arrangements and of sampled magic squares have more and more vertices as $n$ grows (\Cref{tab:vertices}, estimates).
The proof shears the lift onto five levels and reduces to the signs of a table of polynomials in $m=(n-1)/2$.
The same method shows that the square with the two base-$n$ digits exchanged has a gem with $f$-vector $(12,24,14)$ for every odd $n\ge5$ (\Cref{prop:exchanged}); along the way we give a self-contained criterion for a linear square (a square of the uniform step method) to be magic, and count these squares (\Cref{prop:linear-magic}).
More generally, along every family of linear squares with fixed steps and offsets affine in $n$ the vertex count is bounded and the $f$-vector is eventually periodic in $n$ (\Cref{prop:fixed-steps}).
The bound depends on the family: over all linear magic squares of prime order $n$ the largest vertex count grows like $\Theta(\ln n)$ and is at least $\lceil\ln n/(4\ln\varphi)\rceil+2$, with $\varphi$ the golden ratio (\Cref{cor:many-vertices}), and, by computation, for $7\le n\le21$ the Siamese and the digit-exchanged squares are the only linear magic squares, up to $D_4$, with as few as twelve vertices (\Cref{rem:linear}).

The other results require weaker hypotheses than the full magic conditions, and several hold for every arrangement; \Cref{tab:hypotheses} lists what each result needs.
The lifted points balance about the origin for every arrangement, and the heights balance as weights on the grid if and only if the rows and columns have equal sums (\Cref{sec:balance}).
When only the first moments $\sum z\,x$ and $\sum z\,y$ of the heights vanish, the second-moment matrix of the point set is $\operatorname{diag}(a_n,a_n,b_n)$ with $a_n\ne b_n$, and a four-line principal-axes argument shows that the point set has a non-trivial isometry if and only if the arrangement equals its complement up to $D_4$ (\Cref{prop:pointset}).
For the $880$ magic squares of order four up to rotation and reflection we compute all gems in exact integer arithmetic: the isometry group of each gem equals that of its point set, so $48$ gems are centrally symmetric, $304$ have a single half-turn symmetry and $528$ have trivial symmetry (\Cref{prop:iso4}). Across the $880$ squares, the gems have between $8$ and $16$ vertices, realizing $39$ distinct $f$-vectors, $352$ combinatorial types, $155$ exact volumes and $616$ congruence classes (\Cref{prop:congruence}); moreover, the vertex set alone uniquely determines the magic square, for $871$ of the $880$ classes already from the rank of the line-sum equations on the non-vertex cells (\Cref{rem:vertexset}).

Finally, \Cref{sec:magic-vs-random} compares magic gems with the gems of semi-magic squares, of arrangements with vanishing first moments and of random arrangements: whereas combinatorial counts such as the mean numbers of vertices and facets are nearly unchanged across these families, continuous invariants—specifically the volume and the solid centroid—shift systematically as the line-sum constraints are imposed.

\begin{table}[htbp]
\centering
\caption{The results and the hypotheses they need. $^*$Except the eight $D_4$-images of the natural square $s_{ij}=ni+j+1$, whose points are coplanar.}
\label{tab:hypotheses}
\footnotesize
\setlength{\tabcolsep}{4pt}
\begin{tabular}{l l l}
\toprule
result & statement & hypothesis\\
\midrule
\Cref{lem:lines}, \Cref{lem:forced}, \Cref{lem:selfcomp} & centring; forced vertices; complements & any arrangement\\
\Cref{prop:fulldim}, \Cref{prop:subdivisions} & $3$-polytope; faces $\mathcal S(z)$, $\mathcal S(-z)$ & any arrangement$^*$\\
\Cref{lem:cov}, \Cref{prop:pointset} & second moments; point-set isometries & vanishing first moments\\
\Cref{lem:parity} & parity of $V$, $V^\partial$, $V^\circ$; of $E$, $F$ & self-complementary$^*$; associative$^*$\\
\Cref{prop:iso4}, \ref{prop:congruence}, \Cref{rem:vertexset} & order-four census & magic, $n=4$ (computation)\\
\Cref{sec:magic-vs-random} & which invariants see the line conditions & $n=4,5$ (computation)\\
\Cref{prop:linear-magic} & magic linear squares & linear squares, odd $n$\\
\Cref{thm:siamese}, \Cref{prop:exchanged} & Siamese and digit-exchanged gems & two magic squares, odd $n\ge5$\\
\Cref{prop:fixed-steps}, \Cref{cor:many-vertices} & fixed steps; largest vertex count & linear squares; prime $n$\\
\bottomrule
\end{tabular}
\end{table}

\paragraph{Relation to previous work.}
In polyhedral combinatorics, one traditionally studies the space of \emph{all} magic squares of order $n$ simultaneously: nonnegative integer matrices with line sums equal to $t$ are the lattice points of the $t$-th dilate of the Birkhoff polytope \cite{stanley1973magic,ehrhart1973magic,ziegler1995lectures}, magic squares with distinct entries are counted by Ehrhart quasi-polynomials of inside-out polytopes \cite{beck2006insideout,beck2010enumeration}, magic cones admit Hilbert-basis descriptions \cite{ahmed2003polyhedral}, and contingency tables with fixed margins form the corresponding statistical sample spaces \cite{diaconis1995rectangular}.
In all of those constructions, magic squares are individual lattice points inside a high-dimensional polytope in $\R^{n^2}$. Here we take the complementary view, attaching a three-dimensional polytope $\mathcal G(S)\subset\R^3$ to a \emph{single} arrangement $S$ via the regular-subdivision lifting construction of discrete geometry \cite{gelfand1994discriminants,deloera2010triangulations}.

This geometric viewpoint connects naturally with several classical threads. The linear and quadratic balance identities of \Cref{sec:balance} reformulate the mechanical moment invariants of Loly and Rogers \cite{loly2004invariance,rogers2004inertia,rogers2005multipole} in the coordinates of $\mathcal G(S)$, allowing us to settle the Euclidean symmetry and congruence classification of the lifted point sets for every order (\Cref{prop:pointset}). In \Cref{sec:census}, we combine regular subdivision theory with Dudeney's classification of the $880$ order-four magic squares \cite{dudeney1917amusements} to analyse their $352$ combinatorial gem types and explain geometrically why zero-sum lines inflate the volume of the hull. In \Cref{sec:siamese-gems}, we revisit de la Loub\`ere's Siamese square \cite{loubere1693} and the uniform-step (linear) magic squares studied by Lehmer \cite{lehmer1929} and by Apostol and Zuckerman \cite{apostol1951}; combining our five-level affine shear with theorems of Calegari and Walker \cite{calegari2013integer}, Rubin \cite{rubin1970} and Zolotykh \cite{zolotykh2006vertices} on integer hulls of rational polygons yields the $12$-vertex theorem for Siamese squares (\Cref{thm:siamese}), eventual periodicity for fixed-step families (\Cref{prop:fixed-steps}), and the sharp $\Theta(\ln n)$ growth of the maximum vertex count over prime-order linear magic squares (\Cref{cor:many-vertices}).

Sections~\ref{sec:gem}--\ref{sec:discussion} treat, in turn, the gem construction, moment balance, symmetry and congruence, faces and the order-four census, Siamese and linear squares, and open questions. The appendices document the exact computational methods and ancillary verification scripts (\Cref{sec:appendix}) and collect lattice-point counts, vertex statistics, further order-four data and the proofs for \Cref{sec:fixed-steps}.
An interactive 3D display of all $880$ order-four gems is available at \url{https://kylemath.github.io/MagicGemWebpage/}.

\section{The gem of an arrangement}
\label{sec:gem}

Let $S=(s_{ij})$ be an arrangement of order $n\ge2$, with cells indexed by $i,j\in\{0,\dots,n-1\}$ ($i$ the row from the top, $j$ the column from the left), and write $N=n^2$ and $\bar s=(N+1)/2$ for the mean entry.

\begin{definition}[Lifted point set and gem]
\label{def:gem_coords}
To the cell $(i,j)$ assign the point $\mathbf p_{ij}=(x_{ij},y_{ij},z_{ij})\in\R^3$ with
\begin{equation}
x_{ij}=j-\frac{n-1}{2},\qquad y_{ij}=\frac{n-1}{2}-i,\qquad z_{ij}=s_{ij}-\bar s.
\label{eq:gem_coords}
\end{equation}
The \emph{gem} of $S$ is the convex hull $\mathcal G(S)=\operatorname{conv}\{\mathbf p_{ij}\}$, and $z=(z_{ij})$ is the \emph{height function} of $S$; for a magic square we also call it the \emph{Magic Gem}.
\end{definition}

The $x$-axis runs along the columns and the $y$-axis along the rows (top row at positive $y$).
Let $x_j=j-\frac{n-1}2$ and $y_i=\frac{n-1}2-i$ denote the $n$ values taken by $x$ and $y$.
A \emph{line} is one of the $n$ rows, the $n$ columns, the main diagonal $\{i=j\}$ or the anti-diagonal $\{i+j=n-1\}$.
For a line $L$ let $\delta_L(S)=\sum_{c\in L}s_c-\mu_n$ be the \emph{excess} of the line sum over the magic constant, so $S$ is magic iff $\delta_L(S)=0$ for all $2n+2$ lines.
Moment and centroid statements below refer to the $N$ points $\mathbf p_{ij}$ with unit weights, not to the solid $\mathcal G(S)$.

\paragraph{Terminology.}
Each point $\mathbf p_{ij}$ is either a \emph{vertex} of $\mathcal G(S)$, an \emph{interior point} (in the interior of $\mathcal G(S)$), or a \emph{non-vertex boundary point} (on the boundary but not a vertex).
We write $\operatorname{vert}\mathcal G(S)$ for the vertex set, $V$, $V^\circ$, $V^\partial$ for the three numbers ($V+V^\circ+V^\partial=N$), and $(V,E,F)$ for the $f$-vector.
For any set $X\subseteq\R^3$, $\operatorname{Iso}(X)$ is the group of isometries of $\R^3$ that map $X$ onto itself.
The dihedral group $D_4=\langle r,h\rangle$ of the square acts on arrangements by $(gS)_{gc}=s_c$; here $r$ is the quarter-turn, $r^2$ the half-turn, $h$ and $v$ the left--right and top--bottom reflections and $d$ the reflection in the main diagonal $i=j$.
For $g\in D_4$ let $\hat g(x,y,z)=(g(x,y),z)$, where $g$ acts on the plane as the corresponding orthogonal map.
The order-four squares are numbered $\#0,\dots,\#879$ in the lexicographic order of their Fr\'enicle standard forms (the image under $D_4$ with the smallest corner entry at top left and $s_{01}<s_{10}$) \cite{frenicle1693}.

\begin{lemma}[Means, variances, line sums]
\label{lem:lines}
For every arrangement $S$ of order $n$:
\begin{enumerate}
\item $\sum_{i,j}x_{ij}=\sum_{i,j}y_{ij}=\sum_{i,j}z_{ij}=0$ and $\sum_{i,j}x_{ij}y_{ij}=0$;
\item $\frac1N\sum x_{ij}^2=\frac1N\sum y_{ij}^2=a_n:=\frac{n^{2}-1}{12}$ and $\frac1N\sum z_{ij}^2=b_n:=\frac{N^{2}-1}{12}$;
\item for every line $L$, $\delta_L(S)=\sum_{c\in L}z_c$, and
$\sum_{c\in L}\mathbf p_c=\bigl(n\,\bar x_L,\ n\,\bar y_L,\ \delta_L(S)\bigr)$,
where $(\bar x_L,\bar y_L)$ is the centre of the line: $(0,y_i)$ for the row $i$, $(x_j,0)$ for the column $j$, and $(0,0)$ for both diagonals.
\end{enumerate}
In particular $S$ is magic if and only if the vector sum of the $\mathbf p_c$ along every line is horizontal.
\end{lemma}

\begin{proof}
(1) $\sum_{i,j}x_{ij}=n\sum_j x_j=0$, likewise for $y$; $\sum z=\sum_{k=1}^{N}k-N\bar s=0$; and $\sum xy=(\sum_jx_j)(\sum_iy_i)=0$.
(2) $\frac1n\sum_jx_j^2$ is the variance of $n$ equally spaced points with unit spacing, $(n^2-1)/12$; the same holds for $y$, and for $z$ with $N$ points.
(3) Since $\mu_n=n\bar s$, $\sum_{c\in L}z_c=\delta_L$. The cells of row $i$ have $y=y_i$ and $x$ running over all $x_j$; for a column the roles are exchanged; each diagonal is invariant under $(x,y)\mapsto(-x,-y)$.
\end{proof}

\begin{figure}[htbp]
\centering
\includegraphics[width=0.72\textwidth]{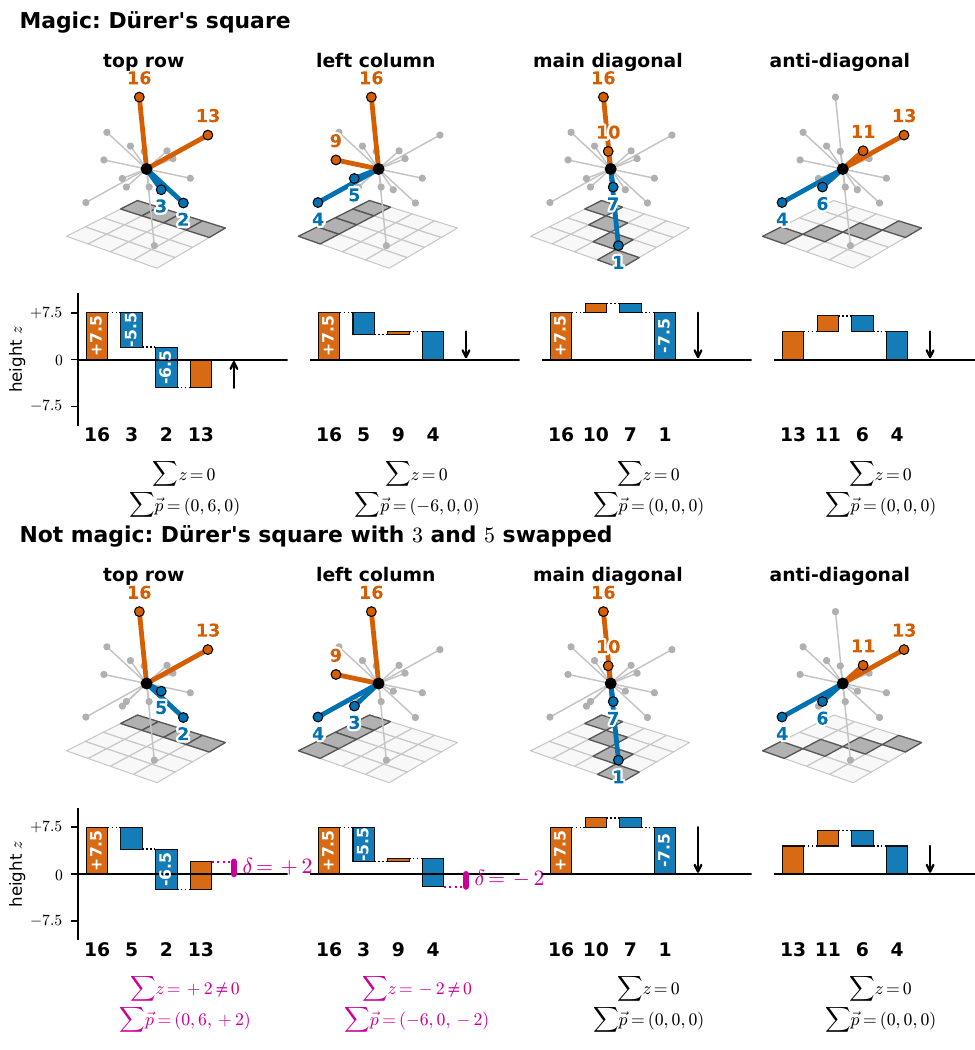}
\caption{Heights cancel along every line. \emph{Top:} for D\"urer's square, the vectors of the top row, left column and both diagonals (coloured; the cells of the line shaded on the floor), and below each the heights $z=s-8.5$ of the line stacked head to tail; the running total returns to $0$, so $\sum_{c\in L}\mathbf p_c$ is horizontal (\Cref{lem:lines}(3)). \emph{Bottom:} after swapping the entries $3$ and $5$ the top row and the left column no longer close (excesses $\delta_L=\pm2$, magenta); the diagonals do not contain the swapped cells.}
\label{fig:lines}
\end{figure}

For $g\in D_4$ the points of $gS$ are $\hat g\,\mathbf p_c$, so
\begin{equation}
\mathcal G(gS)=\hat g\,\mathcal G(S),
\label{eq:d4}
\end{equation}
and since $D_4$ maps magic squares to magic squares, every $D_4$-class of magic squares has a single gem up to congruence.

\begin{proposition}[Dimension]
\label{prop:fulldim}
Let $S$ be an arrangement of order $n\ge2$.
The points $\mathbf p_{ij}$ are coplanar if and only if $S$ is one of the eight $D_4$-images of the natural square $s_{ij}=ni+j+1$.
Otherwise $\mathcal G(S)$ is a $3$-polytope with the origin in its interior.
In particular this holds for every semi-magic square.
\end{proposition}

\begin{proof}
The grid positions are not collinear, so a plane containing all points is not vertical and the points are coplanar if and only if $z=\alpha x+\beta y+\gamma$ on the grid.
Then $u=\alpha$ and $w=-\beta$ are the differences of the entries of horizontally and vertically adjacent cells, so they are integers, and $s_{ij}=s_{00}+uj+wi$.
Since the entries are distinct, $u,w\ne0$; applying $h$, $v$ and $d$ we may assume $1\le u\le w$, so that $s_{00}=1$ is the smallest entry.
The entry $2$ is $1+uj+wi$ with $(i,j)\ne(0,0)$, hence $u=1$, and row $0$ holds $1,\dots,n$.
The entry $n+1$ lies in a row $i\ge1$, so $w\le n$; if $w<n$ then the cells $(0,w)$ and $(1,0)$ both hold $1+w$.
Hence $w=n$ and $s_{ij}=ni+j+1$; conversely, every $D_4$-image of this arrangement has affine heights.
If the points span $\R^3$ affinely, their centroid, which is $\mathbf 0$ by \Cref{lem:lines}(1), is a convex combination with positive weights of all of them and hence an interior point of the hull.
Finally, if $S$ is semi-magic and $z=\alpha x+\beta y+\gamma$, then $\gamma=0$ by \Cref{lem:lines}(1), the row $i$ has height sum $n\beta y_i=0$ for all $i$ and the column $j$ has $n\alpha x_j=0$, so $\alpha=\beta=0$, which is impossible.
\end{proof}

\begin{example}[Lo Shu]
\label{ex:loshu}
For the Lo Shu square $S=\begin{pmatrix}2&7&6\\9&5&1\\4&3&8\end{pmatrix}$, $\bar s=5$ and the nine points are, row by row, $(-1,1,-3)$, $(0,1,2)$, $(1,1,1)$, $(-1,0,4)$, $(0,0,0)$, $(1,0,-4)$, $(-1,-1,-1)$, $(0,-1,-2)$, $(1,-1,3)$.
The origin is the only non-vertex point and is interior; the gem has $f$-vector $(8,16,10)$ and volume $18$ (\Cref{fig:construction}).
\end{example}

\begin{figure}[htbp]
    \centering
    \includegraphics[width=0.9\textwidth]{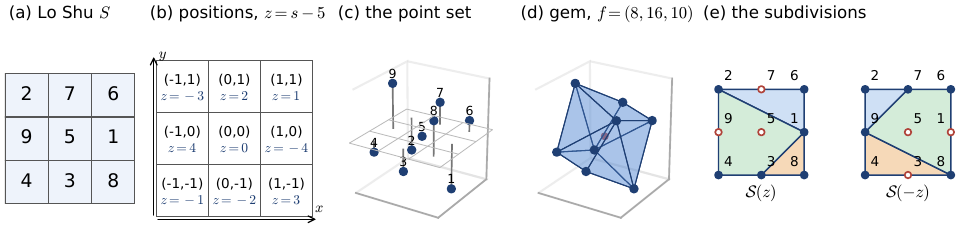}
    \caption{The gem of the Lo Shu square: (a)~the square; (b)~the cell positions $(x_{ij},y_{ij})$, measured from the centre, with the heights $z_{ij}=s_{ij}-5$; (c)~the nine points, labelled by the entries; (d)~their convex hull, with $f$-vector $(8,16,10)$ and the interior point $\mathbf 0$; (e)~the regular subdivisions $\mathcal S(z)$ and $\mathcal S(-z)$ of the grid, the projections of the lower and the upper facets (\Cref{ex:loshu-subdiv}); grid points that belong to no cell are drawn as open circles.}
    \label{fig:construction}
\end{figure}

\section{Balance: the vectors from the centroid}
\label{sec:balance}

Figures~\ref{fig:vectors} and~\ref{fig:lines} draw the points $\mathbf p_{ij}$ as vectors from the origin.
This section records which balance statements these vectors satisfy, and for which arrangements.
The identities are due in substance to Loly and Rogers \cite{loly2004invariance,rogers2004inertia,rogers2005multipole}, who study the moment of inertia and the multipole moments of magic squares and cubes.

\paragraph{Unweighted balance holds for every arrangement.}
By \Cref{lem:lines}(1), $\sum_{i,j}\mathbf p_{ij}=\mathbf 0$ for every arrangement: the origin is the centroid of the $N$ points, and by \Cref{prop:fulldim} it is an interior point of $\mathcal G(S)$ unless $S$ is a $D_4$-image of the natural square.
It is a consequence of centring and concerns the points, not the solid $\mathcal G(S)$, whose centre of mass need not be the origin (Appendix~\ref{sec:further4}).

\begin{figure}[htbp]
\centering
\includegraphics[width=0.44\textwidth]{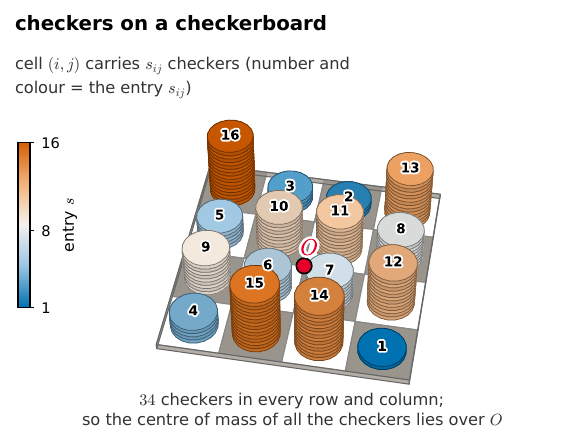}
\caption{D\"urer's square as piles of checkers, $s_{ij}$ on cell $(i,j)$, seen from the front of the board in the orientation of the square; each row and column carries $34$ checkers, so the centre of mass of all the checkers lies over the centre $O$ of the board.}
\label{fig:moments}
\end{figure}

\begin{figure}[htbp]
\centering
\includegraphics[width=0.7\textwidth]{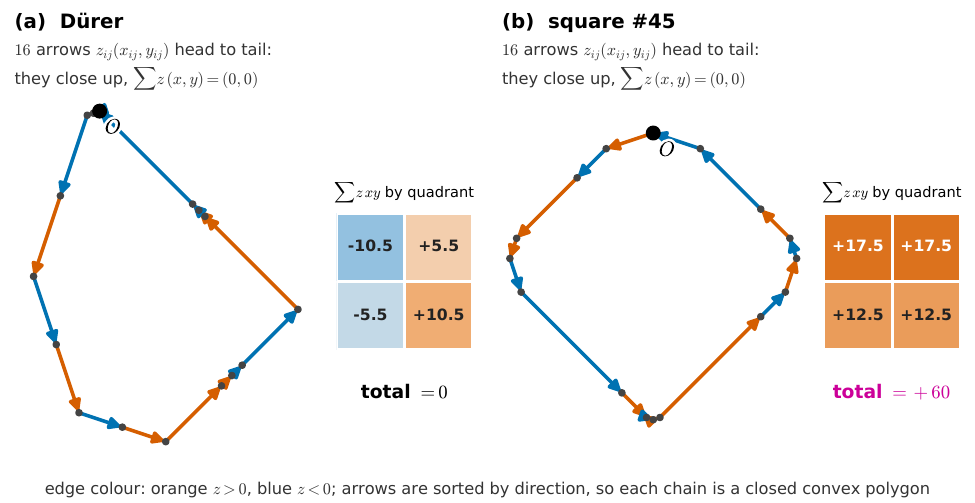}
\caption{The dipole moment $\sum z_{ij}\,(x_{ij},y_{ij})$ for D\"urer's square and square \#45: the $16$ vectors $z_{ij}(x_{ij},y_{ij})$ (orange for $z>0$, blue for $z<0$), sorted by direction and laid head to tail, close up at $O$. The tables give the mixed moment $\sum z\,xy$ over the four quadrants of the grid; the totals are $0$ and $60$.}
\label{fig:weighted}
\end{figure}

\paragraph{Height-weighted balance is the row and column conditions.}
Now give each cell $(x_{ij},y_{ij})$ of the grid the signed weight $z_{ij}$; the total weight is $0$.
Write $R_i=\sum_jz_{ij}$ and $C_j=\sum_iz_{ij}$ for the \emph{row and column excesses}; by \Cref{lem:lines}(3), $S$ is semi-magic if and only if all $R_i$ and $C_j$ vanish.
For every function $\varphi\colon\R\to\R$, grouping the cells by column and by row gives
\begin{equation}
\sum_{i,j}z_{ij}\,\varphi(x_{ij})=\sum_j\varphi(x_j)\,C_j,\qquad
\sum_{i,j}z_{ij}\,\varphi(y_{ij})=\sum_i\varphi(y_i)\,R_i .
\label{eq:moments}
\end{equation}
Since the functions $\varphi(x)$ on the grid span the indicator functions of the columns, $S$ is semi-magic if and only if both sums vanish for every $\varphi$.
For $\varphi(t)=t$ this says that the first moments $\sum z_{ij}x_{ij}$ and $\sum z_{ij}y_{ij}$ of the heights vanish; Rogers and Loly call $\sum z\,(x,y)$ the \emph{dipole moment} \cite{rogers2005multipole}.
Since $\sum x_{ij}=\sum y_{ij}=0$, the dipole moment is also the first moment $\sum s_{ij}(x_{ij},y_{ij})$ of piles of $s_{ij}$ checkers (\Cref{fig:moments}); \Cref{fig:weighted} shows the vectors $z_{ij}(x_{ij},y_{ij})$ closing up.
For $\varphi(t)=t^2$, \eqref{eq:moments} gives Loly's theorem \cite{loly2004invariance}: $\sum_{i,j}s_{ij}(x_{ij}^2+y_{ij}^2)=n^2(n^4-1)/12$ for every semi-magic square.
The vanishing of the first moments ($\varphi(t)=t$) is strictly weaker than the row and column conditions: for arbitrary real heights, Vandermonde invertibility requires the polynomial moments $\varphi(t)=t^m$ to vanish for all $1\le m\le n-1$ to force $R_i=C_j=0$ (for permutations of $\{1,\dots,n^2\}$, integrality and boundedness lower this degree threshold from $n=14$ onward, a problem studied alongside random permutation energies in a companion paper).
Moreover, the diagonal conditions are not of the form \eqref{eq:moments}: they do not determine the mixed moment $\sum z\,xy$, which is $0$ for D\"urer's square and $60$ for square \#45 (\Cref{fig:weighted}; Appendix~\ref{sec:further4}).

\paragraph{Second moments.}
The second moments of the point set are as rigid as the first: when the first moments of the heights vanish they do not depend on the arrangement at all.

\begin{lemma}[Second-moment matrix]
\label{lem:cov}
The second-moment matrix (or empirical covariance matrix) $\Sigma=\frac1N\sum_{i,j}\mathbf p_{ij}\mathbf p_{ij}^{\top}$ of the lifted points of an arrangement $S$ has $\Sigma_{xx}=\Sigma_{yy}=a_n$, $\Sigma_{zz}=b_n$, $\Sigma_{xy}=0$, $N\Sigma_{xz}=\sum_jx_jC_j$ and $N\Sigma_{yz}=\sum_iy_iR_i$. If the first moments $\sum z_{ij}x_{ij}$ and $\sum z_{ij}y_{ij}$ vanish, for instance if $S$ is semi-magic, then $\Sigma=\operatorname{diag}(a_n,a_n,b_n)$, and $a_n<b_n$ for $n\ge2$.
\end{lemma}

\begin{proof}
The diagonal entries $\Sigma_{xx},\Sigma_{yy},\Sigma_{zz}$ and $\Sigma_{xy}=0$ follow immediately from \Cref{lem:lines}(1)--(2), while the off-diagonal moments $N\Sigma_{xz}$ and $N\Sigma_{yz}$ are given by \eqref{eq:moments} with $\varphi(t)=t$.
\end{proof}

Since $b_n=(n^2+1)a_n$, placing equal point masses at the $N$ points $\mathbf p_{ij}$ of an arrangement with vanishing first moments gives an axisymmetric inertia tensor $I=\operatorname{diag}(I_\perp,I_\perp,I_{zz})$ about $\mathbf 0$ with $I_\perp/I_{zz}=(a_n+b_n)/(2a_n)=(n^2+2)/2$ ($9$ for order four), so in torque-free motion the point-weighted gem precesses like a symmetric rotor even when $\mathcal G(S)$ has no geometric symmetry.

\section{Symmetry and congruence}
\label{sec:symmetry}
\label{sec:symmetry-congruence}

\subsection{Point sets of every order}
\label{sec:pointsets}

Let $\tau(x,y,z)=(x,y,-z)$ be the reflection in the plane $z=0$.
The \emph{complement} of an arrangement is $\overline S=N+1-S$, with height function $-z$; $S$ is \emph{self-complementary} if $\overline S=gS$ for some $g\in D_4$, and \emph{associative} if $\overline S=r^2S$, i.e.\ if entries in cells symmetric about the centre add up to $N+1$.
Complements of magic squares are magic.
Since the points of $\overline S$ are $\tau\mathbf p_c$,
\begin{equation}
\mathcal G(\overline S)=\tau\,\mathcal G(S).
\label{eq:complement}
\end{equation}

\begin{lemma}[Elements of $D_4$ that can complement]
\label{lem:selfcomp}
Let $\overline S=gS$ for an arrangement $S$ of order $n\ge2$ and some $g\in D_4$. Then $g$ is unique, $g^2=\mathrm{id}$, and $g$ fixes no cell except possibly the centre cell of an odd grid, whose entry is then $\bar s$.
Consequently $g=r^2$ (so $S$ is associative) if $n$ is odd, and $g\in\{r^2,h,v\}$ if $n$ is even.
The map $\iota_g=\hat g\,\tau$ is an isometry of $\R^3$ that maps the point set of $S$ onto itself, so $\iota_g\in\operatorname{Iso}(\mathcal G(S))$; explicitly $\iota_{r^2}=-\mathrm{id}$ (a point reflection), $\iota_h(x,y,z)=(-x,y,-z)$ and $\iota_v(x,y,z)=(x,-y,-z)$ (half-turns about the $y$- and the $x$-axis).
\end{lemma}

\begin{proof}
If also $\overline S=g'S$ then $g^{-1}g'$ fixes $S$; as the entries are distinct it fixes every cell, and $D_4$ acts faithfully on the cells, so $g=g'$.
Since complementation commutes with $g$, $g^2S=g\overline S=\overline{gS}=\overline{\overline S}=S$, so $g^2=\mathrm{id}$.
If $g$ fixes a cell $c$ then $2s_c=N+1=2\bar s$, so $s_c=\bar s$.
The involutions of $D_4$ are $r^2$, $h$, $v$ and the two diagonal reflections; the diagonal reflections fix $n\ge2$ cells, $h$ and $v$ fix the $n$ cells of the middle column or row when $n$ is odd and none when $n$ is even, and $r^2$ fixes only the centre cell when $n$ is odd.
Finally, the points of $S$ are mapped by $\tau$ to those of $\overline S=gS$, and by $\hat g=\hat g^{-1}$ back to those of $g^2S=S$, by \eqref{eq:d4}.
\end{proof}

\begin{proposition}[Point sets with vanishing first moments]
\label{prop:pointset}
Let $S$ and $T$ be arrangements of order $n\ge2$ with $\sum z_{ij}x_{ij}=\sum z_{ij}y_{ij}=0$ (the first moments of the heights vanish), for instance semi-magic squares, and let $P_S$, $P_T$ be their point sets.
\begin{enumerate}
\item Every isometry $\psi$ of $\R^3$ with $\psi(P_S)=P_T$ is of the form $\hat g$ or $\hat g\,\tau$ with $g\in D_4$; $\hat g(P_S)=P_T$ if and only if $T=gS$, and $\hat g\,\tau(P_S)=P_T$ if and only if $T=g\overline S$.
\item $P_S$ and $P_T$ are congruent if and only if $T\in D_4S\cup D_4\overline S$.
\item $\operatorname{Iso}(P_S)=\{\mathrm{id}\}$ unless $S$ is self-complementary, and $\operatorname{Iso}(P_S)=\{\mathrm{id},\iota_g\}$ if $\overline S=gS$.
\end{enumerate}
\end{proposition}

\begin{proof}
An isometry of $\R^3$ mapping one finite set onto another maps centroid to centroid, and both centroids are $\mathbf 0$ (\Cref{lem:lines}(1)); so $\psi$ is orthogonal, and $\psi\Sigma_S\psi^{\top}=\Sigma_T$ for the second-moment matrices.
By \Cref{lem:cov}, $\Sigma_S=\Sigma_T=\operatorname{diag}(a_n,a_n,b_n)$ with $a_n\neq b_n$, so $\psi$ commutes with $\Sigma_S$ and preserves its eigenspaces, the plane $z=0$ and the $z$-axis: $\psi=\operatorname{diag}(B,\varepsilon)$ with $B\in O(2)$ and $\varepsilon=\pm1$.
Since $\psi=\operatorname{diag}(B,\varepsilon)$ commutes with orthogonal projection onto the plane $z=0$, and the projection of both point sets is the $n\times n$ grid, $B$ is a symmetry of the grid, that is, an element $g\in D_4$; hence $\psi=\hat g$ or $\psi=\hat g\tau$.
By \eqref{eq:d4}, $\hat g(P_S)=P_{gS}$ and $\hat g\tau(P_S)=P_{g\overline S}$, and the point set determines the arrangement; this proves (1), and (2) follows.
For (3) take $T=S$: $\hat g(P_S)=P_S$ forces $gS=S$, hence $g=\mathrm{id}$, and $\hat g\tau(P_S)=P_S$ forces $g\overline S=S$, i.e.\ $\overline S=g^{-1}S=gS$ by \Cref{lem:selfcomp}.
\end{proof}

Notice that the diagonal line sums play no role in \Cref{prop:pointset}: vanishing horizontal first moments alone lock the principal axes of the point set to the grid.

\subsection{The isometry groups of the order-four gems}
\label{sec:iso4}

For the gem, $\operatorname{Iso}(\mathcal G(S))=\operatorname{Iso}(\operatorname{vert}\mathcal G(S))$, since isometries preserve extreme points, and $\operatorname{Iso}(P_S)\subseteq\operatorname{Iso}(\mathcal G(S))$; the inclusion is an equality when every point is a vertex, but in general the gem might have isometries that do not preserve the non-vertex points.
For order four this does not happen.

\begin{proposition}[Isometry groups of the order-four gems; exact computation]
\label{prop:iso4}
For each of the $880$ $D_4$-classes of magic squares $S$ of order $4$, $\operatorname{Iso}(\mathcal G(S))=\operatorname{Iso}(P_S)$. Hence:
\begin{enumerate}
\item for the $48$ associative squares, $\operatorname{Iso}(\mathcal G(S))=\{\mathrm{id},-\mathrm{id}\}$;
\item for the $304$ self-complementary squares that are not associative, $\operatorname{Iso}(\mathcal G(S))=\{\mathrm{id},\iota_g\}$ with $g\in\{h,v\}$, a half-turn about a horizontal coordinate axis;
\item for the remaining $528$ squares, $\operatorname{Iso}(\mathcal G(S))$ is trivial.
\end{enumerate}
In particular a gem has an improper isometry if and only if $S$ is associative, and the $832$ gems in cases (2) and (3) are chiral (chiral: having no orientation-reversing isometry).
\end{proposition}

\begin{proof}
The squares were sorted into the three cases by comparing $\overline S$ with the eight images $gS$; by \Cref{lem:selfcomp} the complementing element is $r^2$, $h$ or $v$.
For each square we computed $\operatorname{Iso}(\operatorname{vert}\mathcal G(S))$ in exact integer arithmetic (Appendix~\ref{sec:appendix}); its order is $2$ in cases (1) and (2) and $1$ in case (3), equal to the order of $\operatorname{Iso}(P_S)$ by \Cref{prop:pointset}(3), and $\operatorname{Iso}(P_S)\subseteq\operatorname{Iso}(\mathcal G(S))$.
The maps $\iota_h$, $\iota_v$ are proper.
\end{proof}

The partition (1)--(3) is classical: the associative squares form Dudeney's group III and the self-complementary non-associative ones his group VI \cite{dudeney1917amusements}.
What \Cref{prop:iso4} adds is that the hulls have no further symmetry.
\Cref{tab:symmetry} summarises the three cases and \Cref{fig:symmetry} shows one gem of each.

\begin{table}[htbp]
\centering
\caption{Symmetry types of the $880$ order-four gems; ``classes'' counts congruence classes.}
\label{tab:symmetry}
\footnotesize
\begin{tabular}{l c r l r l}
\toprule
type of $S$ & Dudeney & squares & $\operatorname{Iso}(\mathcal G)$ & classes & $V$\\
\midrule
associative & III & $48$ & $\{\pm\mathrm{id}\}$ & $48$ & $10,12,14,16$\\
self-complementary, not associative & VI & $304$ & $\{\mathrm{id},\iota_g\}$, half-turn & $304$ & $8,10,12,14,16$\\
not self-complementary & others & $528$ & trivial & $264$ & $8,\dots,16$\\
\midrule
total & & $880$ & & $616$ &\\
\bottomrule
\end{tabular}
\end{table}

\begin{figure}[htbp]
\centering
\includegraphics[width=0.8\textwidth]{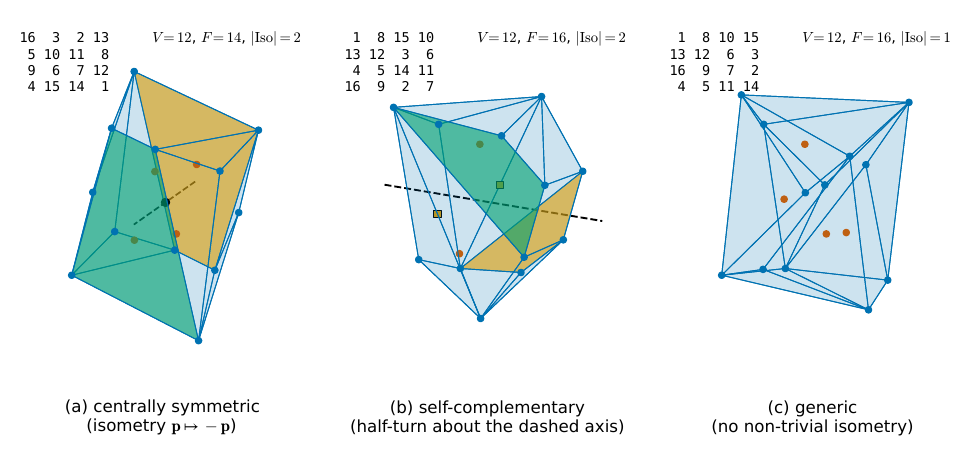}\\[2pt]
\includegraphics[width=0.8\textwidth]{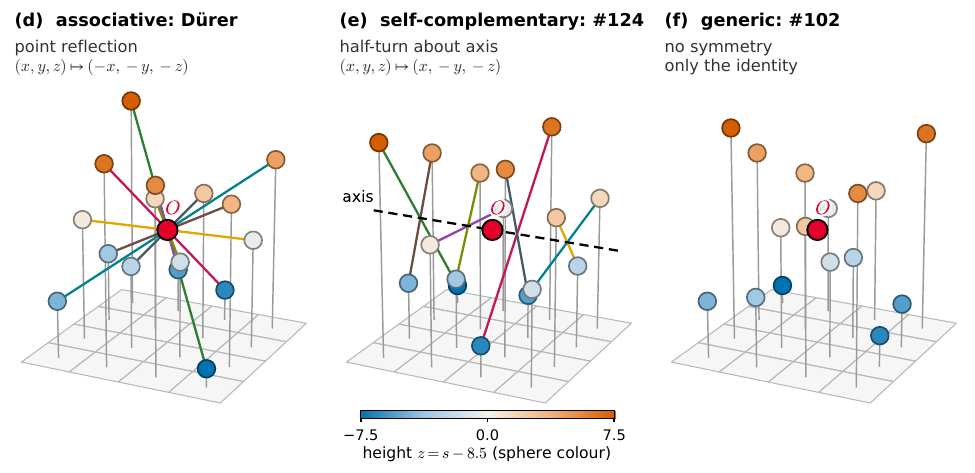}
\caption{Gems (top) and point sets (bottom) of three order-four squares, all drawn from the same viewpoint, with the square printed above each gem. Interior points: red discs; non-vertex boundary points: orange squares. (a,\,d)~D\"urer's square, associative: a pair of opposite facets (orange, green), and the $8$ antipodal pairs $\{\mathbf p,-\mathbf p\}$ joined through $O$. (b,\,e)~Square \#124, self-complementary and not associative: the half-turn about the dashed axis exchanges the highlighted facets and the $8$ pairs of points. (c,\,f)~Square \#102: no non-trivial isometry; the gem is chiral and its mirror image in $z=0$ is the gem of $17-S$.}
\label{fig:symmetry}
\end{figure}

\begin{lemma}[Parity]
\label{lem:parity}
Let $S$ be a self-complementary arrangement of order $n$ that is not a $D_4$-image of the natural square (for instance, a self-complementary semi-magic square). Then $V$ and $V^\partial$ are even and $V^\circ\equiv n\pmod2$. If $S$ is associative, $E$ and $F$ are even as well.
\end{lemma}

\begin{proof}
Let $\overline S=gS$. Then $N+1-s_{gc}=s_c$, i.e.\ $z_{gc}=-z_c$, so $\iota_g\mathbf p_c=\mathbf p_{gc}$ for every cell $c$.
The involution $\iota_g$ is an isometry of the $3$-polytope $\mathcal G(S)$ (\Cref{prop:fulldim}), so it permutes the vertices, the interior points and the non-vertex boundary points.
By \Cref{lem:selfcomp} it fixes $\mathbf p_c$ only if $gc=c$, i.e.\ only for the centre cell of an odd grid, whose point is the interior point $\mathbf 0$.
If $S$ is associative, $\iota_g=-\mathrm{id}$, and a face fixed by $-\mathrm{id}$ would contain the centroid $\mathbf 0$ of its vertices, which is interior; so $-\mathrm{id}$ pairs the edges and the facets.
\end{proof}

The hypothesis is needed: the natural square of order $3$ is associative, but its gem is a parallelogram and has no interior points.
For $n=4$ the $48$ associative gems have $V^\circ\in\{0,2,4\}$ ($4$, $23$ and $21$ gems), and among the $528$ squares that are not self-complementary all four combinations of the parities of $V$ and $V^\circ$ occur.

\begin{remark}[Facet orbits of the half-turns]
\label{rem:facet-orbits}
The half-turn $\iota_h$ (the case $\iota_v$ is the same after the reflection $d$) maps the plane $\mathbf a\cdot\mathbf p=b$ to $\iota_h(\mathbf a)\cdot\mathbf p=b$. The two coincide if and only if $\iota_h(\mathbf a)=\mathbf a$, or $\iota_h(\mathbf a)=-\mathbf a$ and $b=0$; a plane through the interior point $\mathbf 0$ cannot support the gem. So an $\iota_h$-invariant facet plane is one of the vertical planes $y=\pm\frac{n-1}2$ over the top and bottom rows, which by \Cref{prop:subdivisions}(2) below carry facets unless the row is an arithmetic progression, and $F$ need not be even. Among the $304$ order-four gems of case (2) of \Cref{prop:iso4}, $282$ have two fixed facets, $18$ one (so $F$ is odd) and $4$ none.
\end{remark}

\begin{remark}[Congruent facets]
\label{rem:congruent-facets}
An isometry of a gem maps facets to congruent facets, so the facets of the $352$ self-complementary gems of order four come in congruent pairs, apart from the facets fixed by a half-turn (\Cref{rem:facet-orbits}), none of which is congruent to another facet.
These isometries are diagonal matrices, so the pairs stay congruent when the heights are rescaled, as in printed models and paper nets.
Similarly, since $-\mathrm{id}\,\mathcal G(S)=\mathcal G(r^2\overline S)$, every upper facet of $\mathcal G(S)$ is parallel and congruent to a lower facet of the half-turned complement $\mathcal G(r^2\overline S)$, allowing copies of a gem and its complement twin to stack face-to-face into periodic columns; for instance, in square $\#762$ the entries $\{13,14,15,16\}$ lie on a centred knight's-move square of area $5$ and are coplanar ($z=6-\frac15(3x-y)$), forming a quadrilateral upper facet matched by $\{1,2,3,4\}$ below.
Accidental congruences between facets not related by an isometry of $\mathcal G(S)$ also occur among the order-four gems ($182$ of the $528$ asymmetric gems and $28$ of the $352$ symmetric gems have at least one such pair), though unlike symmetry-forced congruences most of these accidental matches disappear when the vertical coordinate is rescaled relative to the horizontal grid.
\end{remark}

\begin{remark}[Repeating chains of facets]
\label{rem:belts}
The pairing of \Cref{rem:congruent-facets} respects adjacency: an isometry of the gem maps facets that share an edge to facets that share an edge, so a chain of adjacent facets is matched by a chain of congruent facets in the same order.
This is easiest to see on belts.
If a plane through $\mathbf 0$ contains no point $\mathbf p_c$, the facets that it crosses, taken in order along the boundary of the section of the gem by the plane, form a closed chain in which consecutive facets share an edge.
If an isometry of the gem maps the plane to itself and acts on it as a half-turn about $\mathbf 0$, it preserves the cyclic order of the edges of the section and fixes none of them, so it maps the chain to itself shifted half way round: the second half of the chain repeats the first. If it acts on the plane as a reflection, it reverses the chain.
So for an associative square the chain along the plane $z=0$ of the mean entry repeats after half a turn, and so does, for the other self-complementary squares, the chain along the vertical plane through $\mathbf 0$ perpendicular to the axis of the half-turn; along $z=0$ the half-turn reverses its chain, which then reads the same in both directions.
For order four these chains have $8$ or $10$ facets for the associative gems, and $4$ to $14$ facets around the axis for the others.
\end{remark}

\section{Faces of the gem and a census of order four}
\label{sec:census}

After two facts about the faces of every gem, this section reports a computational census of the gems of the $880$ magic squares of order four (vertices, faces, volumes, combinatorial types, congruence classes) and compares them with the gems of other arrangements.

\subsection{Forced vertices and regular subdivisions}
\label{sec:forced}

The \emph{sides} of the grid are its first and last rows and columns.
Every gem lies in the box $\mathcal B_n=[-\frac{n-1}2,\frac{n-1}2]^2\times[-\frac{N-1}2,\frac{N-1}2]$ and meets each of its six faces: the top and the bottom face in the points of the entries $N$ and $1$, and the vertical face over a side $e$ of the grid in the polygon $\operatorname{conv}\{\mathbf p_c: c\in e\}$.
Each contact is a face of $\mathcal G(S)$, so its vertices are vertices of $\mathcal G(S)$; the lemma below names those that can be read off from the entries.

\begin{lemma}[Forced vertices]
\label{lem:forced}
Let $S$ be an arrangement of order $n\ge2$. The following points are vertices of $\mathcal G(S)$; we call them the \emph{forced vertices} of $S$: the points over the four corner cells; the points carrying the entries $1$ and $N$; and, for each side, the points over the cells of the smallest and of the largest entry on that side.
\end{lemma}

\begin{proof}
A point that is the unique maximiser of a linear functional over the $N$ points is a vertex.
The entries are distinct, so $\mp z$ have unique maximisers where $s=1$ and $s=N$.
The functional $-x+y$ attains its maximum only at the top-left corner, and similarly for the other corners.
On the left column $x=-\frac{n-1}2$, while $x\ge-\frac{n-1}2+1$ elsewhere; for $\varepsilon=1/N$ the functional $-x+\varepsilon z$ exceeds $\frac{n-1}2-\frac12$ on the left column and is less than $\frac{n-1}2-\frac12$ elsewhere, so its unique maximiser is the point of largest height on the left column; $-x-\varepsilon z$ gives the smallest, and the other sides follow by symmetry.
\end{proof}

For order four the forced vertices number between $4$ and $13$, and in general other points are vertices as well (\Cref{sec:census-4}).

\paragraph{Gems and regular subdivisions.}
\label{sec:subdivisions}
Geometrically, a three-dimensional gem $\mathcal G(S)$ is bounded over the square $\operatorname{conv}A=[-\frac{n-1}2,\frac{n-1}2]^2$ by a piecewise-linear concave upper roof $z=U(x,y)$ and a piecewise-linear convex lower floor $z=L(x,y)$, joined around the perimeter of the grid by vertical walls wherever the upper and lower boundary chains do not coincide.
Viewing the gem from below ($z\to-\infty$) or from above ($z\to+\infty$) projects its lower and upper facets onto two \emph{regular subdivisions} of the $n\times n$ grid $A$ of cell positions (\Cref{fig:construction}e shows both subdivisions for the Lo Shu).
More precisely, for any height function $\omega\colon A\to\R$, the regular subdivision $\mathcal S(\omega)$ of $A$ is obtained by projecting to the plane the lower faces of $\operatorname{conv}\{(c,\omega(c)):c\in A\}$, i.e.\ the faces whose outer normals have negative vertical coordinate; its cells are the subsets of grid points whose lifts lie on a common lower face \cite[Ch.~2]{deloera2010triangulations}, \cite[Ch.~7]{gelfand1994discriminants}.
(Each cell of $\mathcal S(\omega)$ is recorded as a marked subset of $A$, since the convex hull of a cell may contain grid points whose lifted heights lie strictly above the supporting plane and therefore do not belong to the cell.)
The subdivision $\mathcal S(\omega)$ is unchanged if $\omega$ is multiplied by a positive scalar or if an affine function of $(x,y)$ is added to it; the height functions that induce a given subdivision form a relatively open polyhedral cone, its \emph{secondary cone}.
Because the upper faces of $\mathcal G(S)$ are the reflections in $z=0$ of the lower faces of the lift by $-z$, the lower and upper boundaries of $\mathcal G(S)$ project to $\mathcal S(z)$ and $\mathcal S(-z)$.

\begin{proposition}[Gems as pairs of regular subdivisions]
\label{prop:subdivisions}
Let $S$ be an arrangement of order $n\ge2$ that is not a $D_4$-image of the natural square.
\begin{enumerate}
\item Every facet of $\mathcal G(S)$ is a lower facet, an upper facet or a vertical facet. Projection to the plane maps the lower facets onto the maximal cells of $\mathcal S(z)$ and the upper facets onto the maximal cells of $\mathcal S(-z)$.
\item The vertical facets lie in the four planes $x=\pm\frac{n-1}2$, $y=\pm\frac{n-1}2$; the plane over a side of the grid contains a facet if and only if the entries along that side, in cell order, are not an arithmetic progression.
\item $\mathbf p_c$ is a vertex of $\mathcal G(S)$ if and only if $c$ is a vertex of a cell of $\mathcal S(z)$ or of $\mathcal S(-z)$.
\item The face lattice of $\mathcal G(S)$, with vertices labelled by cells, is determined by the pair $(\mathcal S(z),\mathcal S(-z))$, hence by the secondary cones containing $z$ and $-z$. The complement $N+1-S$ exchanges the two subdivisions.
\end{enumerate}
\end{proposition}

\begin{proof}
(1) follows directly from the definition of $\mathcal S(z)$ and, after reflecting in $z=0$, of $\mathcal S(-z)$.
For (2), a vertical facet lies in a supporting plane $\ell(x,y)=\max_A\ell$ for some non-zero linear form $\ell$; the maximisers of $\ell$ on the grid $A$ are either a single corner or the $n$ cells of a side, and the resulting face is two-dimensional if and only if we are in the second case and the lifted points over that side are not collinear—that is, the heights along the side are not an affine function of the position.
For (3), the vertical line through any vertex $\mathbf p_c$ meets $\mathcal G(S)$ in a segment having $\mathbf p_c$ as an endpoint, so $\mathbf p_c$ lies on a lower or an upper face, of which it is a vertex; conversely, the vertices of any face are vertices of the polytope.
For (4), consider the upper concave roof $U$ and lower convex floor $L$ of $\mathcal G(S)$ over $\operatorname{conv}A$.
The vertical thickness $U-L$ is concave, non-negative, and strictly positive at the projection of the interior point $\mathbf 0$ (\Cref{prop:fulldim}), so $U>L$ throughout $\operatorname{int}\operatorname{conv}A$; thus a lower face and an upper face can meet only over the boundary $\partial\operatorname{conv}A$.
Over each side $e$ of the grid, $\mathcal G(S)$ meets the vertical boundary plane in the polygon $\operatorname{conv}\{\mathbf p_c: c\in e\}$, which is bounded above and below by the restrictions of $U$ and $L$ to $e$; this polygon is a two-dimensional vertical facet if and only if those two boundary chains do not coincide into a single segment, i.e.\ by (2) when the entries along $e$ are not an arithmetic progression.
Finally, each lower or upper facet projects bijectively onto its cell in $A$, so its vertices are the lifts of the extreme points of that cell.
Hence the marked cells of $\mathcal S(z)$ and $\mathcal S(-z)$, together with the vertical facets, determine all vertex--facet incidences, and replacing $S$ by its complement $N+1-S$ negates $z$ and swaps $\mathcal S(z)$ with $\mathcal S(-z)$.
\end{proof}

\begin{example}[Lo Shu, continuing \Cref{ex:loshu}]
\label{ex:loshu-subdiv}
Writing each grid point as its entry, $\mathcal S(z)$ has the maximal cells $\{1,2,6\}$, $\{1,3,8\}$, $\{1,2,3,4\}$, and $\mathcal S(-z)$ has $\{2,7,9\}$, $\{4,8,9\}$, $\{6,7,8,9\}$.
The vertical facets are the triangles $\{2,4,9\}$, $\{2,6,7\}$, $\{1,6,8\}$, $\{3,4,8\}$.
The centre point (height $0$) lies in the convex hull of each quadrilateral cell but belongs to no cell: the plane of $\{1,2,3,4\}$ is $z=-x-y-3$, which is $-3<0$ at the centre.
\end{example}

For the $880$ order-four gems, $856$ have four vertical facets, $20$ have three and $4$ have two; the height functions of the $7040$ order-four magic squares induce $417$ distinct regular subdivisions of the $4\times4$ grid up to $D_4$, and for $548$ of the $880$ classes $\mathcal S(z)$ is a triangulation.

\subsection{The census}
\label{sec:census-4}

There are $8$ magic squares of order $3$, all images of the Lo Shu under $D_4$, and $7{,}040$ of order $4$, forming $880$ $D_4$-classes; the order-four squares were listed by Fr\'enicle de Bessy \cite{frenicle1693} and have been recounted by computer \cite{beck2010enumeration}.
Dudeney sorted them into twelve groups according to the pattern of complementary pairs $(a,17-a)$ \cite{dudeney1917amusements} (Appendix~\ref{sec:appendix} records the group labels for all $7040$ squares).
The number $275{,}305{,}224$ of magic squares of order five up to $D_4$ is due to Schroeppel (1973) \cite{gardner1976}; for order six see \cite{pinn1998number,mino2024order6,trump_howmany}.

For each of the $880$ squares we computed $\mathcal G(S)$ exactly (integer facet inequalities, rational volumes; Appendix~\ref{sec:appendix}).
\Cref{tab:census} and \Cref{fig:gem-geometry-880} summarise the results.

\begin{table}[htbp]
\centering
\caption{The $880$ order-four gems by number $V$ of vertices: numbers of squares, of combinatorial types (face lattices up to isomorphism), of $f$-vectors and of distinct exact volumes, and the volume as a fraction of the box $\mathcal B_4$ ($\operatorname{vol}\mathcal B_4=135$). A volume can occur for several values of $V$, so the volume counts add up to more than $155$. Bottom: numbers of gems with $k$ interior points ($V^\circ=k$) and with $k$ non-vertex boundary points ($V^\partial=k$).}
\label{tab:census}
\footnotesize
\begin{tabular}{l rrrrrrrrr r}
\toprule
$V$ & 8 & 9 & 10 & 11 & 12 & 13 & 14 & 15 & 16 & all\\
\midrule
squares & 7 & 10 & 103 & 80 & 380 & 106 & 162 & 14 & 18 & 880\\
types   & 4 & 1 & 27 & 32 & 129 & 52 & 85 & 7 & 15 & 352\\
$f$-vectors & 4 & 1 & 6 & 5 & 7 & 5 & 5 & 3 & 3 & 39\\
volumes & 4 & 5 & 51 & 33 & 111 & 40 & 65 & 6 & 14 & 155\\
\addlinespace[2pt]
min of $\operatorname{vol}/\operatorname{vol}\mathcal B_4$ & .622 & .588 & .481 & .549 & .430 & .456 & .452 & .498 & .449 & .430\\
mean of $\operatorname{vol}/\operatorname{vol}\mathcal B_4$ & .642 & .623 & .631 & .630 & .634 & .614 & .608 & .560 & .536 & .623\\
max of $\operatorname{vol}/\operatorname{vol}\mathcal B_4$ & .648 & .667 & .701 & .699 & .721 & .688 & .691 & .615 & .598 & .721\\
\bottomrule
\end{tabular}

\medskip
\begin{tabular}{l rrrrrrr}
\toprule
$k$ & 0 & 1 & 2 & 3 & 4 & 5 & 6\\
\midrule
gems with $V^\circ=k$ & 37 & 32 & 353 & 178 & 280 & 0 & 0\\
gems with $V^\partial=k$ & 357 & 204 & 257 & 40 & 21 & 0 & 1\\
\bottomrule
\end{tabular}
\end{table}

\paragraph{Vertices, faces and volumes.}
Between $8$ and all $16$ points are vertices; $V=12$ is the most frequent value ($380$ gems), and the mean is $269/22\approx12.23$.
The $39$ $f$-vectors range from $(8,14,8)$ (one gem) to $(16,36,22)$ (nine gems); the most frequent is $(12,26,16)$ ($164$ gems).
The volume takes $155$ distinct values, from $58$ to $292/3$, with mean $84.05$; its denominators are $1,2,3,6$, because for even $n$ all coordinates are half-odd integers, so that $\mathcal G(S)+(\frac12,\frac12,\frac12)$ is a lattice polytope.
The Ehrhart polynomials of these lattice polytopes take $554$ distinct values (Appendix~\ref{sec:ehrhart}).

\paragraph{Combinatorial types.}
Two gems have the same combinatorial type if their vertex--facet incidence graphs are isomorphic by an isomorphism that maps vertices to vertices and facets to facets.
Testing pairs of incidence graphs for isomorphism gives $352$ types (Appendix~\ref{sec:appendix}).
By \eqref{eq:complement} the gems of $S$ and $\overline S$ are mirror images, so every type is a union of complement orbits; the $93$ types with one member are self-complementary squares, and of the $201$ types with two members, $177$ are mirror pairs and $24$ consist of two self-complementary squares.
The numbers of types with $1,2,3,\dots$ members are
\[
\begin{array}{r|rrrrrrrrrrrrrr}
\text{squares per type} & 1&2&3&4&5&6&7&8&10&13&14&16&26&32\\\hline
\text{types} & 93&201&8&20&4&7&4&5&5&1&1&1&1&1
\end{array}
\]
The largest type, with $f=(10,22,14)$ (twelve triangles and two quadrilaterals), has $32$ members.

\begin{proposition}[Congruence classes of the order-four gems; exact computation]
\label{prop:congruence}
Let $S$ and $T$ be magic squares of order four. Then $\mathcal G(S)$ and $\mathcal G(T)$ are congruent if and only if $T\in D_4S\cup D_4\overline S$. There are $616$ congruence classes of gems: $352$ self-complementary classes and $264$ mirror pairs.
\end{proposition}

\begin{proof}
``If'' is \eqref{eq:d4} and \eqref{eq:complement}. For ``only if'': the sorted multiset of pairwise squared distances of the vertex set is a congruence invariant, and it takes $616$ distinct values over the $880$ gems (exact computation). Gems in one complement orbit are congruent, and there are $616$ orbits ($352$ self-complementary classes and $264$ pairs). Hence the congruence classes are exactly the complement orbits.
\end{proof}

\begin{remark}[The vertex set determines the square]
\label{rem:vertexset}
Remarkably, the vertices of an order-four magic gem suffice to reconstruct the entire magic square: if $S$ and $T$ are magic squares of order four and $\hat g\operatorname{vert}\mathcal G(S)=\operatorname{vert}\mathcal G(T)$ for some $g\in D_4$, then $T=gS$.
To see why, note that by \eqref{eq:d4}, $gS$ and $T$ share the same vertex set, so they agree on all vertex cells and leave the same subset $U$ of non-vertex cells undetermined.
The ten magic line-sum conditions provide a system of linear equations for the entries on $U$, governed by the $0/1$ incidence matrix $A_U$ between the ten lines and the cells of $U$. Whenever $\operatorname{rank}A_U=|U|$, the missing entries are uniquely determined by linear algebra alone.
Direct calculation shows that $\operatorname{rank}A_U=|U|$ for $871$ of the $880$ $D_4$-classes.
For the remaining nine classes ($\#45$, $\#46$, $\#47$, $\#49$, $\#50$, $\#54$, $\#165$, $\#166$ and $\#672$), enumerating all order-four magic squares that agree with $S$ on its vertex cells reveals that five classes admit no other completion, while $\#45$, $\#47$, $\#50$ and $\#54$ admit $1$, $1$, $7$ and $7$ alternative magic completions respectively (for instance, $\#109$ agrees with $\#45$ on its twelve vertices). In every such case, however, the alternative completion turns at least one cell of $U$ into a new vertex of the hull, so no two distinct squares share the exact same vertex set.
\end{remark}

\paragraph{Forced vertices and the geometry of the extreme squares.}
For order four nothing beyond the forced vertices of \Cref{lem:forced} is a vertex of every gem: the only cells that carry a vertex in all $880$ gems are the four corners, and the only entries that always label a vertex are $1$ and $16$.
In eight gems the point of the entry $2$ is a non-vertex boundary point, for instance for square \#593 below.
The gems of the Lo Shu, of D\"urer's square and of the Siamese square of order $5$ have $8$ of $9$, $12$ of $16$ and $12$ of $25$ points as vertices and volumes $18$, $92$ and $830/3$.

The three extreme order-four gems highlighted in \Cref{fig:gem-geometry-880}(d)--(f) illustrate how the placement of the entries governs the polyhedral geometry:
\begin{itemize}\setlength\itemsep{1pt}
\item \emph{Minimum vertices and facets ($\#54$, $f=(8,14,8)$):} In square $\#54$ below, the top row $(1,6,11,16)$ and the bottom row $(4,7,10,13)$ are both arithmetic progressions. By \Cref{prop:subdivisions}(2), the lifted points along each of these two sides of the grid are collinear, eliminating both the front and back vertical facets and making $\#54$ one of the four order-four gems with only two vertical facets and the unique gem achieving the minimal $f$-vector $(8,14,8)$.
\item \emph{All sixteen points as vertices ($\#802$, $V=16$):} The gem of D\"urer's square has $12$ boundary vertices and its $4$ central cells in the interior. Exchanging its first two rows and its last two rows produces square $\#802$, which moves the extreme entries $\{1,2,15,16\}$ from the boundary into the interior $2\times2$ core so that all four interior points pierce through the boundary canopy and all $16$ points become vertices.
\item \emph{Maximum volume ($\#379$, $\operatorname{vol}=292/3$):} In square $\#379$, the four smallest entries $\{1,2,3,4\}$ and four largest entries $\{13,14,15,16\}$ alternate along the boundary and diagonals so that the upper concave surface $U$ is propped near the ceiling across the grid while the lower convex surface $L$ is pulled near the floor, maximizing the enclosed volume.
\end{itemize}
The squares cited in Figures~\ref{fig:weighted}, \ref{fig:symmetry} and~\ref{fig:gem-geometry-880} and in the text are
{\small\[
\setlength{\arraycolsep}{2pt}
\#45=\begin{pmatrix}1&6&11&16\\8&15&2&9\\12&3&14&5\\13&10&7&4\end{pmatrix}\quad \#54=\begin{pmatrix}1&6&11&16\\15&12&5&2\\14&9&8&3\\4&7&10&13\end{pmatrix}\quad \#102=\begin{pmatrix}1&8&10&15\\13&12&6&3\\16&9&7&2\\4&5&11&14\end{pmatrix}\quad \#124=\begin{pmatrix}1&8&15&10\\13&12&3&6\\4&5&14&11\\16&9&2&7\end{pmatrix}
\]
\[
\setlength{\arraycolsep}{2pt}
\#379=\begin{pmatrix}2&12&5&15\\13&7&10&4\\16&6&11&1\\3&9&8&14\end{pmatrix}\quad \#593=\begin{pmatrix}4&1&16&13\\5&15&2&12\\14&8&9&3\\11&10&7&6\end{pmatrix}\quad \#802=\begin{pmatrix}5&10&11&8\\16&3&2&13\\4&15&14&1\\9&6&7&12\end{pmatrix}
\]}

\begin{figure}[htbp]
\centering
\includegraphics[width=0.76\textwidth]{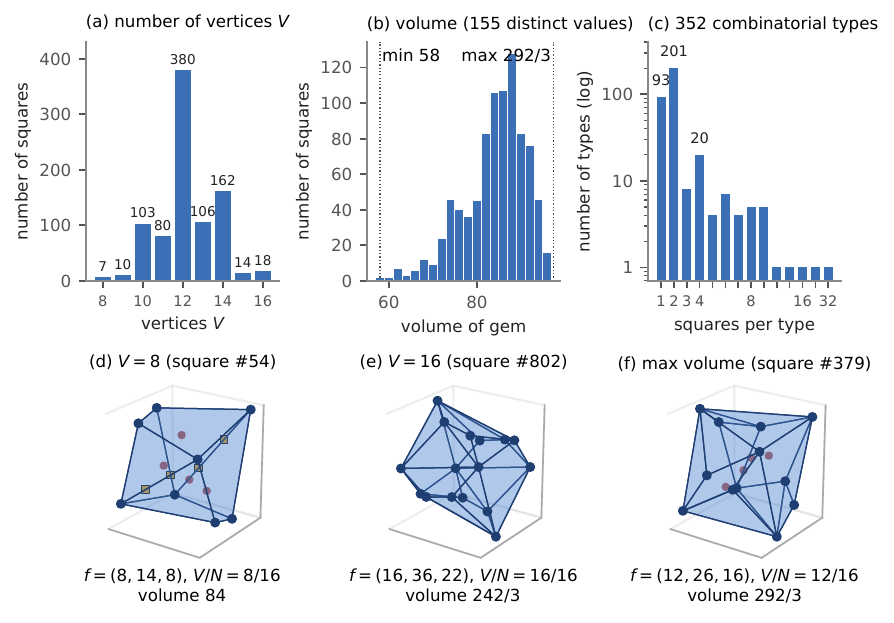}
\caption{The $880$ order-four gems. (a)~Number of vertices. (b)~Exact volume. (c)~Sizes of the $352$ combinatorial types (logarithmic scale). (d)--(f)~Three extreme gems with translucent faces: the unique gem with $f=(8,14,8)$ (square \#54; interior points as red discs, non-vertex boundary points as orange squares), a gem with all $16$ points as vertices (square \#802), and the gem of largest volume (square \#379).}
\label{fig:gem-geometry-880}
\end{figure}

\subsection{Which invariants see the line conditions}
\label{sec:magic-vs-random}

To understand how the line-sum conditions shape the geometry of the gem, we compare four nested families of arrangements of order four (\Cref{tab:magic-vs-random}): all $16!\approx2.09\cdot10^{13}$ arrangements, the $6{,}697{,}177{,}424$ arrangements whose heights have vanishing first moments $\sum z\,x=\sum z\,y=0$ (about one in $3124$), the $549{,}504$ semi-magic squares, and the $7040$ magic squares.

\begin{table}[htbp]
\centering
\caption{Mean invariants of the gems of order four across four nested families. The first three columns are estimates from $10^5$ uniform samples each (standard errors at most $0.006$ for the counts and fractions, $0.024$ for the volume and $0.001$ for the centroid): arbitrary arrangements, arrangements with vanishing first moments $\sum z\,x=\sum z\,y=0$ (rejection sampling), and semi-magic squares (uniform over all $549{,}504$). The magic column is exact over all $7040$ magic squares. The solid centroid is the centre of mass of the solid polytope $\mathcal G(S)$.}
\label{tab:magic-vs-random}
\scriptsize
\setlength{\tabcolsep}{6pt}
\begin{tabular}{l rrrr}
\toprule
 & random & first moments $0$ & semi-magic & magic\\
invariant & \multicolumn{3}{c}{(estimates, $10^5$ samples each)} & (exact)\\
\midrule
vertices $V$ & 12.055 & 12.076 & 12.124 & 12.227\\
interior points $V^\circ$ & 2.439 & 2.508 & 2.625 & 2.718\\
non-vertex boundary points $V^\partial$ & 1.506 & 1.415 & 1.251 & 1.055\\
lower facets & 6.503 & 6.503 & 6.455 & 6.439\\
volume & 75.75 & 81.72 & 82.68 & 84.05\\
distance of the solid centroid from $\mathbf 0$ & 0.348 & 0.353 & 0.159 & 0.109\\
fraction with $\operatorname{Iso}(\mathcal G)\ne\{\mathrm{id}\}$ & 0.000 & 0.000 & 0.101 & 0.400\\
\bottomrule
\end{tabular}

\end{table}

As \Cref{tab:magic-vs-random} shows, the combinatorial counts are almost completely insensitive to the line-sum constraints.
The mean number of vertices rises only slightly from $12.06$ for random arrangements and $12.12$ for semi-magic squares to $12.23$ for magic squares, while the mean numbers of lower and upper facets differ by less than $0.08$ across all four families and the mean number of vertical facets stays between $3.97$ and $4.00$.
Even the small difference in vertex count between magic and semi-magic squares is largely an artifact of symmetry: self-complementary magic squares have more vertices on average ($12.43$ versus $12.09$ for asymmetric magic squares, with the $48$ associative squares averaging $11.88$), and self-complementary squares constitute $40\%$ of the magic squares compared with $10\%$ of the semi-magic squares. Restricting to non-self-complementary squares reduces the gap between magic and semi-magic mean vertex counts to $0.017\pm0.004$.

By contrast, the continuous geometric invariants respond noticeably as each successive balance condition is imposed:
\begin{enumerate}\setlength\itemsep{1pt}
\item \emph{Vanishing first moments} align the principal axes of the solid with the coordinate axes (reducing the off-diagonal tilt $(\Sigma_{xz}^2+\Sigma_{yz}^2)^{1/2}$ of the solid's second-moment matrix from $0.73$ to $0.24$) and increase the mean volume from $75.75$ to $81.72$, without yet pulling the solid centroid toward the origin.
\item \emph{Equal row and column sums} (the semi-magic condition) draw the solid centroid more than halfway toward the origin (from distance $0.35$ to $0.16$) and raise the mean volume to $82.68$.
\item \emph{Equal main diagonal sums} increase the mean volume further to $84.05$ and pull the mean centroid distance down to $0.11$, an effect that persists even when self-complementary squares are excluded.
\end{enumerate}
There is a clean geometric reason why zero-sum lines inflate the volume of the hull. By \Cref{prop:subdivisions}(4), $\operatorname{vol}\mathcal G(S)=\iint(U-L)$, where $U$ is the concave upper boundary function and $L$ is the convex lower boundary function of the gem. Along any line $L_0$ of the grid on which $\sum_{c\in L_0}z_c=0$, the inequalities $L(c)\le z_c\le U(c)$ force
\[
\sum_{c\in L_0}L(c)\;\le\;0\;\le\;\sum_{c\in L_0}U(c).
\]
Consequently, $z$ cannot be uniformly convex or uniformly concave along $L_0$: a bowl-shaped or dome-shaped height profile (which would make $U$ and $L$ hug each other closely and produce a thin, small-volume hull) is forbidden along every zero-sum line, forcing the upper canopy $U$ and lower bowl $L$ apart.

A similar pattern appears in heuristic explorations at order five. Comparing $2\cdot10^4$ certified exact hulls each of uniform random arrangements, random arrangements with vanishing first moments, and order-five semi-magic and magic squares generated by simulated annealing (Appendix~\ref{sec:vertexstats}; note that the sampling distribution of the annealer at $n=5$ is not known to be uniform), the mean vertex counts are nearly identical across all four families ($15.40$, $15.42$, $15.36$ and $15.39$, with standard errors $0.012$), whereas the mean volumes climb steadily ($250.1$, $260.9$, $265.4$ and $268.5$) and the mean distance of the solid centroid from the origin drops from $0.47$ to $0.24$ and $0.22$.

\section{Larger orders: Siamese and linear squares}
\label{sec:siamese-gems}

Beyond order four exhaustive computation is out of reach ($275{,}305{,}224$ magic squares of order five up to $D_4$), and samples suggest that the number of vertices grows with $n$ (\Cref{tab:vertices}, estimates).
The Siamese square behaves in the opposite way (\Cref{thm:siamese}, the main result); its proof uses a closed form, the invariance of $\mathcal S(\pm z)$ under adding affine functions (\Cref{sec:subdivisions}), the forced vertices and \Cref{prop:pointset}, and it also applies to the digit-exchanged square (\Cref{prop:exchanged}); other linear squares have many more vertices (\Cref{rem:linear}).

\subsection{Linear squares}

The \emph{Siamese} square of odd order $n=2m+1$ is de la Loub\`ere's classical construction \cite{loubere1693}: it has $1$ in the middle cell of the top row; each next entry is placed one step up and one step to the right, cyclically, or one step down if that cell is occupied.
For $n=5$ it is
\[
\begin{pmatrix}17&24&1&8&15\\23&5&7&14&16\\4&6&13&20&22\\10&12&19&21&3\\11&18&25&2&9\end{pmatrix},
\]
and for $n=3$ it is a $D_4$-image of the Lo Shu.
Write $\langle t\rangle\in\{0,\dots,n-1\}$ for the residue of $t$ modulo $n$.
The Siamese square is a \emph{linear} (cyclic) square
\[
s(c)=n\langle\mathbf u\cdot c+\alpha\rangle+\langle\mathbf w\cdot c+\beta\rangle+1,\qquad c=(i,j),\quad \mathbf u=(u_1,u_2),\ \mathbf w=(w_1,w_2)\in\Z_n^2,\ \alpha,\beta\in\Z_n,
\]
with $\mathbf u=(1,1)$, $\mathbf w=(1,2)$, $\alpha=-m$, $\beta=-2m$ (Step~1 of the proof of \Cref{thm:siamese}).
If $c\mapsto(\langle\mathbf u\cdot c+\alpha\rangle,\langle\mathbf w\cdot c+\beta\rangle)$ is a bijection of $\Z_n^2$, the cell of the entry $nk_1+k_0+1$ is an affine function of its digits $(k_1,k_0)$, so the linear squares are the squares of the uniform step method: consecutive entries advance by a fixed step, and after every $n$ entries by a fixed break-move \cite[Ch.~I]{andrews1960magic}.
When this method fills the square and when it gives equal row and column sums was analysed by Lehmer \cite{lehmer1929} (see \cite{travassos2013}), and the method was studied further by Apostol and Zuckerman \cite{apostol1951}; we give a self-contained proof of the following criterion, including the diagonals, and of the count.

\begin{proposition}[Linear magic squares]
\label{prop:linear-magic}
Let $n$ be odd. The linear square $s$ is a magic square if and only if
\begin{enumerate}
\item[(a)] $u_1$, $u_2$, $w_1$, $w_2$ and $u_1w_2-u_2w_1$ are units of $\Z_n$, and
\item[(b)] for each digit, $(a_1,a_2,\gamma)=(u_1,u_2,\alpha)$ and $(w_1,w_2,\beta)$, and each pair $(\sigma,\tau)=(a_1+a_2,\gamma)$ (main diagonal) and $(a_1-a_2,\gamma-a_2)$ (anti-diagonal): if $d=\gcd(\sigma,n)>1$, then $\tau\equiv\frac{d-1}2\pmod d$.
\end{enumerate}
Consequently there are $\varphi(n)^2\sum f(r)f(r')$ linear magic squares of order $n$, the sum over the units $r,r'$ of $\Z_n$ with $r-r'$ a unit, where $f(r)=n/(\gcd(r+1,n)\gcd(r-1,n))$, and one eighth as many $D_4$-classes.
For a prime $n=p$ this is $(p-1)^3(p^3-6p^2+10p-2)/8$ classes.
If $3\mid n$, then for every unit $r$ one of $r\pm1$ is divisible by $3$, so that $f(r)\le n/3$.
\end{proposition}

\begin{proof}
Write $A(c)=\langle\mathbf u\cdot c+\alpha\rangle$ and $B(c)=\langle\mathbf w\cdot c+\beta\rangle$.
The map $c\mapsto(A(c),B(c))$ is affine on $\Z_n^2$ with linear part of determinant $u_1w_2-u_2w_1$, so $s$ is a bijection onto $\{1,\dots,N\}$ if and only if the determinant is a unit.
Along a line $c=c_0+ke$, $k=0,\dots,n-1$, the digit $A$ takes the values $\langle\sigma k+\tau\rangle$ with $\sigma=\mathbf u\cdot e$ and $\tau=\mathbf u\cdot c_0+\alpha$.
If $d=\gcd(\sigma,n)$, these run $d$ times through the $n/d$ residues $t,t+d,\dots,t+n-d$, where $t\in\{0,\dots,d-1\}$ and $t\equiv\tau\pmod d$.
So $\sum_kA=d\bigl(\frac ndt+d\binom{n/d}2\bigr)=\frac{n(n-1)}2+n\,e_A$ with $e_A=t-\frac{d-1}2$ and $|e_A|\le\frac{d-1}2<n$, and likewise $\sum_kB=\frac{n(n-1)}2+n\,e_B$.
The line sum of $s$ is $n\sum A+\sum B+n=\mu_n+n^2e_A+ne_B$; it equals $\mu_n$ only if $n\mid e_B$, i.e.\ $e_B=0$, and then $e_A=0$.
So a line has sum $\mu_n$ if and only if, for each digit, $\sigma$ is a unit or $\tau\equiv\frac{d-1}2\pmod d$.
For the row $i$, $e=(0,1)$, $\sigma=u_2$ and $\tau=u_1i+\alpha$; if $d=\gcd(u_2,n)>1$, the condition for the rows $i=0,1$ forces $d\mid u_1$, and then $d$ divides the determinant.
So, given bijectivity, all rows have sum $\mu_n$ if and only if $u_2$ and $w_2$ are units, and all columns ($e=(1,0)$, $\sigma=u_1$) if and only if $u_1$ and $w_1$ are units.
The main diagonal has $c_0=(0,0)$, $e=(1,1)$, and the anti-diagonal $c_0=(0,n-1)$, $e=(1,-1)$, so $\tau\equiv\alpha-u_2$ there; this proves (a) and (b).
For the count write $u_1=ru_2$ and $w_1=r'w_2$ with units $r,r'$; the determinant is $u_2w_2(r-r')$.
With $d_1=\gcd(r+1,n)$ and $d_2=\gcd(r-1,n)$, (b) prescribes $\alpha$ modulo $d_1$ and modulo $d_2$; $\gcd(d_1,d_2)$ divides $\gcd(2,n)=1$, so exactly $n/(d_1d_2)=f(r)$ values of $\alpha$ are admissible, and likewise $f(r')$ values of $\beta$.
Summing over $u_2,w_2,r,r'$ gives the count; $D_4$ acts freely on arrangements and maps linear magic squares to linear magic squares, hence the factor $1/8$.
For $n=p$ prime, $f(r)=p$ unless $r=\pm1$, and $f(\pm1)=1$; so $\sum_rf(r)=(p-1)(p-2)$, $\sum_rf(r)^2=(p-3)p^2+2$, and the sum over $r\ne r'$ is $(p-1)^2(p-2)^2-(p-3)p^2-2=(p-1)(p^3-6p^2+10p-2)$.
If $3\mid n$ and $r$ is a unit, then $3\nmid r$, so $3$ divides $r-1$ or $r+1$.
\end{proof}

\begin{corollary}[Two linear magic squares]
\label{cor:two-linear}
For odd $n=2m+1$, both
\[
s(i,j)=n\langle i+j-m\rangle+\langle i+2j-2m\rangle+1
\quad\text{and}\quad
s'(i,j)=n\langle i+2j-2m\rangle+\langle i+j-m\rangle+1
\]
are magic squares. We call $s'$ the \emph{digit-exchanged square}.
\end{corollary}

\begin{proof}
The digits have $(a_1,a_2,\gamma)=(1,1,-m)$ and $(1,2,-2m)$; their coefficients are units, as $n$ is odd, and the determinant is $\pm1$.
For the first, the main diagonal has the unit step $\sigma=2$, and the anti-diagonal $\sigma=0$, $d=n$, $\tau=-m-1\equiv m=\frac{n-1}2$.
For the second, the anti-diagonal has $\sigma=-1$, and the main diagonal $\sigma=3$, a unit unless $3\mid n$; then $m\equiv1\pmod 3$ and $\tau=-2m\equiv1=\frac{3-1}2\pmod3$.
\end{proof}

For $n=5,7,11$ the formula gives $184$, $3159$ and $89{,}125$ classes.
For $n=9$ the units $r\equiv1\pmod3$ have $f(r)=1,3,3$ ($r=1,4,7$), those with $r\equiv2$ have $f=3,3,1$ ($r=2,5,8$), and $r-r'$ is a unit if and only if $r\not\equiv r'\pmod 3$; so there are $6^2\cdot2\cdot7^2=3528$ linear magic squares and $441$ classes.
This explains the small number at $n=9$: when $3\mid n$ every linear magic square has, for each digit, a diagonal along which the digit is not a permutation of $\Z_n$, and the offsets are restricted accordingly.

\subsection{The Siamese theorem}

Both proofs below end with the following certificate, which is also the one applied to every computed hull (Appendix~\ref{sec:appendix}).

\begin{lemma}[Closing a list of facets]
\label{lem:closing}
Let $Q\subset\R^3$ be finite and $W\subseteq Q$, and let $\mathbf a_k\cdot\mathbf p\le b_k$ ($k=1,\dots,K$, $K\ge2$, with pairwise distinct planes) hold at every point of $Q$.
Suppose that each plane $\mathbf a_k\cdot\mathbf p=b_k$ contains three non-collinear points of $W$, and that every edge of each polygon $F_k=\operatorname{conv}\{\mathbf w\in W:\mathbf a_k\cdot\mathbf w=b_k\}$ is an edge of some other $F_l$.
Then $\operatorname{conv}Q=\operatorname{conv}W=\{\mathbf p:\mathbf a_k\cdot\mathbf p\le b_k\text{ for all }k\}$, its facets are $F_1,\dots,F_K$, and its vertices are the vertices of these polygons.
\end{lemma}

\begin{proof}
Let $\Pi=\operatorname{conv}W$ and let $H$ be the solution set of the inequalities, so $\Pi\subseteq\operatorname{conv}Q\subseteq H$.
Each plane supports $\Pi$ and meets it in the $2$-dimensional face $F_k$; as there are $K\ge2$ distinct such planes, $\Pi$ is a $3$-polytope and the $F_k$ are distinct facets.
An edge of $F_k$ is an edge of $\Pi$ and lies in exactly two facets of $\Pi$, $F_k$ and, by hypothesis, some $F_l$; so the list is closed under adjacency across edges.
The facet-adjacency graph of a $3$-polytope is connected, so $F_1,\dots,F_K$ are all the facets of $\Pi$, and $\Pi=H$.
\end{proof}

\begin{theorem}[Siamese gems]
\label{thm:siamese}
Let $S$ be the Siamese square of odd order $n=2m+1\ge3$.
\begin{enumerate}
\item[(a)] For $n=3$ (the Lo Shu up to $D_4$) the vertices of $\mathcal G(S)$ are the eight boundary points, $f=(8,16,10)$ (\Cref{ex:loshu}), and $\operatorname{Iso}(\mathcal G(S))=\{\pm\mathrm{id}\}$.
\item[(b)] For $n\ge5$ the vertices of $\mathcal G(S)$ are exactly the twelve forced vertices of \Cref{lem:forced}, the points over the cells
\[
\begin{array}{lll}
(0,0)\colon \bar s+m+2, & (0,m-1)\colon N-1, & (0,m)\colon 1, \\
(0,2m)\colon \bar s+m, & (m-1,0)\colon N-m, & (m,0)\colon m+2, \\
(m,2m)\colon N-m-1, & (m+1,2m)\colon m+1, & (2m,0)\colon \bar s-m, \\
(2m,m)\colon N, & (2m,m+1)\colon 2, & (2m,2m)\colon \bar s-m-2,
\end{array}
\]
listed with their entries; $f=(12,26,16)$, with $12$ triangular and $4$ quadrilateral facets; and $\operatorname{Iso}(\mathcal G(S))=\operatorname{Iso}(P_S)=\{\pm\mathrm{id}\}$.
\end{enumerate}
\end{theorem}

So only $12$ of the $n^2$ points are vertices, and the combinatorial type of the gem does not depend on $n\ge5$.

\begin{proof}
\emph{Step 1: closed form.}
We show that $S$ is the square $s$ of \Cref{cor:two-linear}.
The map $\kappa(i,j)=(\langle i+j-m\rangle,\langle i+2j-2m\rangle)$ is a bijection of $\Z_n^2$, and $\kappa(0,m)=(0,0)$.
Write $k-1=nk_1+k_0$ with $0\le k_0<n$; we prove by induction on $k$ that the entry $k$ is placed in the cell $\kappa^{-1}(k_1,k_0)$.
This holds for $k=1$. Suppose it holds for $1,\dots,k$, with $k$ in the cell $c=\kappa^{-1}(k_1,k_0)$; the occupied cells are $\kappa^{-1}(k_1',k_0')$ with $nk_1'+k_0'<k$.
The up-right move $(i,j)\to(i-1,j+1)$ keeps $i+j$ and increases $i+2j$ by $1$, so it leads to $\kappa^{-1}(k_1,\langle k_0+1\rangle)$.
If $k_0<n-1$, this cell is empty and is the cell prescribed for $k+1$.
If $k_0=n-1$, i.e.\ $n\mid k$, it is the cell $\kappa^{-1}(k_1,0)$ of the entry $nk_1+1$, so the move is blocked; the downward move $(i,j)\to(i+1,j)$ increases both $i+j$ and $i+2j$ by $1$ and leads to the empty cell $\kappa^{-1}(k_1+1,0)$, the cell prescribed for $k+1=n(k_1+1)+1$.
By \Cref{cor:two-linear}, $S$ is magic. It is also associative: by the closed form, $s(i,j)+s(2m-i,2m-j)=N+1$.

\emph{Step 2: the forced vertices.}
In every row and every column the leading digit $\langle i+j-m\rangle$ takes each value in $\{0,\dots,n-1\}$ once, so the largest entry of a side is in the cell where this digit is $n-1$ and the smallest where it is $0$.
On the top row $i=0$ this gives $j=m-1$ with entry $n(n-1)+\langle-2\rangle+1=N-1$, and $j=m$ with entry $1$; on the bottom row $i=2m$, $j=m$ with entry $n(n-1)+n=N$, and $j=m+1$ with entry $\langle n+1\rangle+1=2$; on the left column $j=0$, $i=m-1$ with entry $n(n-1)+\langle-m-1\rangle+1=N-m$, and $i=m$ with $\langle-m\rangle+1=m+2$; on the right column $j=2m$, $i=m$ with entry $n(n-1)+\langle3m\rangle+1=N-m-1$, and $i=m+1$ with $\langle3m+1\rangle+1=m+1$.
The corner entries follow in the same way.
The entries $1$ and $N$ are among the side extremes, and for $m\ge2$ the eight side cells are distinct and are not corners; so the forced vertices are the twelve listed points. For $m=1$ the side extremes are the corners and the midpoints of the sides.

\emph{Proof of (a).} For $n=3$ the forced vertices are the eight boundary points, and the only other point is the interior point $\mathbf 0$ (\Cref{ex:loshu}). An isometry of $\mathcal G(S)$ maps its vertex set onto itself and so fixes the centroid $\mathbf 0$ of the vertex set; hence it preserves $P_S$, and $\operatorname{Iso}(P_S)=\{\pm\mathrm{id}\}$ by \Cref{prop:pointset}, since $S$ is associative.
From now on let $m\ge2$. It remains to show that no point other than the twelve forced vertices is a vertex.

\emph{Step 3: shear to five levels.}
Why should the convex hull of $n^2$ points have only $12$ vertices for arbitrarily large $n$? The geometric key is that in the closed form $s(i,j)=n\langle i+j-m\rangle+\langle i+2j-2m\rangle+1$, the modular remainders differ from the linear functions $\ell_1=x-y+m$ and $\ell_2=2x-y+m$ (where $x=j-m$ and $y=m-i$) only by multiples of $n$ that record how many times each coordinate wraps around the torus. Since subtracting an affine function of $(x,y)$ and rescaling vertically is an affine shear that preserves the face lattice and vertex cells of the hull (\Cref{prop:subdivisions}(3)), we can subtract the global linear ramp $\lambda(x,y)=n\ell_1+\ell_2+1-\bar s$ from $z$. This collapses all $n^2$ lifted points onto just five flat horizontal terraces!

Specifically, with $\langle t\rangle=t-n\lfloor t/n\rfloor$ the closed form gives
\begin{gather*}
z=\lambda(x,y)+n\,\zeta(x,y),\qquad \lambda(x,y)=n(x-y+m)+(2x-y+m)+1-\bar s,\\
\zeta(x,y)=-n\Big\lfloor\tfrac{\ell_1}{n}\Big\rfloor-\Big\lfloor\tfrac{\ell_2}{n}\Big\rfloor,\qquad \ell_1=x-y+m,\quad \ell_2=2x-y+m=\ell_1+x .
\end{gather*}
The shear $\Phi(x,y,z)=(x,y,(z-\lambda(x,y))/n)$ maps $\mathcal G(S)$ onto $\Phi(\mathcal G(S))=\operatorname{conv}\{(x,y,\zeta(x,y))\}$, which has the same face lattice and vertex cells as $\mathcal G(S)$ (equivalently, $\mathcal S(\pm z)=\mathcal S(\pm\zeta)$).
For $|x|,|y|\le m$ we have $-m\le\ell_1\le3m$ and $-2m\le\ell_2\le4m$, so both wrap-around quotients $\lfloor\ell_1/n\rfloor$ and $\lfloor\ell_2/n\rfloor$ lie in $\{-1,0,1\}$.
If $\ell_1<0$, then $x\le y-m-1\le-1$, so $\ell_2<0$ and $\zeta=n+1$.
If $\ell_1\ge n$, then $x\ge y+m+1\ge1$, so $\ell_2\ge n$ and $\zeta=-(n+1)$.
If $0\le\ell_1\le2m$, then $\lfloor\ell_2/n\rfloor=-1,0,1$ according as $y\ge2x+m+1$, $|y-2x|\le m$, or $y\le2x-m-1$.
Conversely, if $|y-2x|\le m$, then $\ell_1<0$ is impossible, since it would force $\ell_2<0$, i.e.\ $y\ge2x+m+1$; likewise $\ell_1\ge n$ would force $y\le2x-m-1$.
Thus, across the grid $|x|,|y|\le m$, the sheared height $\zeta$ takes only five values, partitioning the grid into five polygonal terraces (\emph{level polygons}; \Cref{fig:siamese-layers}a):
$\zeta=n+1$ on $\Lambda_{+2}=\{y\ge x+m+1\}=\operatorname{conv}\{(-m,m),(-1,m),(-m,1)\}$;
$\zeta=1$ on $\Lambda_{+1}=\{2x+m+1\le y\le x+m\}=\operatorname{conv}\{(-m,0),(-m,1-m),(-1,m-1)\}$;
$\zeta=0$ on $\Lambda_{0}=\{|y-2x|\le m\}=\operatorname{conv}\{(-m,-m),(0,m),(m,m),(0,-m)\}$;
and $\zeta=-1$, $-(n+1)$ on $\Lambda_{-1}=-\Lambda_{+1}$, $\Lambda_{-2}=-\Lambda_{+2}$.

\emph{Step 4: the hull.}
Because every sheared grid point on a level polygon $\Lambda_k$ is a convex combination of the corners of $\Lambda_k$ at the same height $\zeta$, only the $3+3+4+3+3=16$ corners of the five level polygons can be extreme points of $\Phi(\mathcal G(S))$.
Let $q_1,\dots,q_{12}$ be the images under $\Phi$ of the twelve forced vertices and $q_{13},\dots,q_{16}$ the remaining four corners of the level polygons; their coordinates are listed in \Cref{tab:slacks}.
Let $H$ be the polyhedron defined by the eight inequalities below and their images $(1')$--$(8')$ under $-\mathrm{id}$.
Inequality $(k)$ is the plane through the first three points of its list below (its normal is the cross product of two edge vectors):
\begin{align*}
&(1)\ \zeta\le 2m+2, \qquad (2)\ y\le m, \qquad (3)\ -x\le m, \qquad (8)\ -2y+\zeta\le 2m,\\
&(4)\ -2(m{+}1)x+(2m{+}1)y-m\zeta\le m(2m{+}1),\qquad (7)\ 2mx-y+m\zeta\le m(2m{-}1),\\
&(5)\ 2(m{+}1)x-2(m{+}1)y+(2m{+}1)\zeta\le 2m(m{+}1),\\
&(6)\ 2m(m{+}2)x-(2m^2{+}2m{-}1)y+m(2m{+}1)\zeta\le m(2m^2{+}2m{-}1).
\end{align*}
\Cref{tab:slacks} lists the slack (right-hand side minus left-hand side) of $(1)$--$(8)$ at the sixteen points $q_k$.
Every non-zero entry is a product of positive constants, of linear factors ($m-1$, $m$, $m+1$, $m+2$, $2m-1$, $2m+1$, $2m+3$, $4m-3$, $4m-1$), which are positive for $m\ge2$, and of the quadratics $2m^2+2m+1$, $2m^2+2m-1$, $4m^2+2m-1$, $4m^2-2m+1$, $4m^2+m+1$, $4m^2-3m+1$, $m^2+m+1$.
Each of these quadratics is positive at $m=1$ and has its minimum at some $m\le3/8$, hence is positive for all $m\ge1$.
The set $\{q_1,\dots,q_{16}\}$ is centrally symmetric, so all sixteen inequalities hold at all sixteen points and hence, by convexity, at all points $(x,y,\zeta)$, each of which lies in a level polygon.
The zeros of the table show that the planes $(1)$--$(8)$ contain the following $q_k$ ($k\le12$), in cyclic order:
\[
\begin{array}{llll}
(1)\colon q_1q_2q_3, & (2)\colon q_1q_2q_8q_7, & (3)\colon q_1q_3q_5q_4, & (4)\colon q_1q_4q_7,\\
(5)\colon q_2q_3q_6, & (6)\colon q_2q_6q_9, & (7)\colon q_2q_8q_9, & (8)\colon q_3q_5q_6,
\end{array}
\]
with the antipodal lists for $(1')$--$(8')$ ($q_{10}=-q_1$, $q_{11}=-q_3$, $q_{12}=-q_2$, $q_9=-q_4$, $q_8=-q_5$, $q_7=-q_6$).
All four points of each quadrilateral are forced vertices, hence in convex position, so only the cyclic order needs checking.
In the plane $y=m$, written in coordinates $(x,\zeta)$, $q_1=(-m,2m+2)$, $q_2=(-1,2m+2)$, $q_8=(m,0)$ and $q_7=(0,0)$ form a trapezoid whose parallel sides $q_1q_2$ and $q_8q_7$ are traversed in opposite directions.
In the plane $x=-m$, written in coordinates $(y,\zeta)$, the line through $q_5=(-m,0)$ and $q_1=(m,2m+2)$ has height $(m+1)(y+m)/m$.
This is $m+1>1$ at $y=0$, so $q_4=(0,1)$ lies below the line, and $(m+1)^2/m<2m+2$ at $y=1$, so $q_3=(1,2m+2)$ lies above it.
Hence $q_1q_5$ is a diagonal and $q_1q_3q_5q_4$ is the cyclic order.
(The non-vertex point $q_{13}=(-m,1-m,1)$ lies on plane $(3)$ inside this quadrilateral.)
These sixteen polygons ($12$ triangles and the $4$ quadrilaterals $(2),(3),(2'),(3')$) have $26$ edges, each shared by exactly two of them.
By \Cref{lem:closing}, with $Q$ the sheared points and $W=\{q_1,\dots,q_{12}\}$, $\Phi(\mathcal G(S))=H$ has exactly these sixteen facets and the vertices $q_1,\dots,q_{12}$; mapping back by $\Phi^{-1}$, the gem has the twelve forced vertices and $f$-vector $(12,26,16)$.

\emph{Step 5: isometries.}
$S$ is associative (Step~1), so $\operatorname{Iso}(P_S)=\{\pm\mathrm{id}\}$ by \Cref{prop:pointset}.
An isometry of $\mathcal G(S)$ fixes the centroid $\mathbf 0$ of its centrally symmetric vertex set and is orthogonal.
Write $\tilde q_k=\Phi^{-1}(q_k)$ for the forced vertices in the original coordinates:
$\tilde q_1=(-m,m,m+2)$, $\tilde q_2=(-1,m,2m^2+2m-1)$, $\tilde q_3=(-m,1,2m^2+m)$, $\tilde q_4=(-m,0,-2m^2-m+1)$, $\tilde q_5=(-m,-m,-m)$, $\tilde q_6=(0,-m,2m^2+2m)$, and the other six are their negatives.
The squared norms
$3m^2<3m^2+4m+4<4m^4+4m^3-2m^2-2m+1<4m^4+4m^3+2m^2+1<4m^4+8m^3+m^2-4m+2<4m^4+8m^3+5m^2$
of $\tilde q_5,\tilde q_1,\tilde q_4,\tilde q_3,\tilde q_2,\tilde q_6$ are distinct for $m\ge2$, so an isometry maps each $\tilde q_k$ to $\pm\tilde q_k$.
The inner products of $\tilde q_1$ with $\tilde q_2,\dots,\tilde q_6$ are $2m^3+7m^2+4m-2$, $m(2m^2+6m+3)$, $-(2m^3+4m^2+m-2)$, $-m(m+2)$, $m(2m^2+5m+4)$, all non-zero, so all signs agree with that of $\tilde q_1$, and the isometry is $\pm\mathrm{id}$.
\end{proof}

\begin{table}[htbp]
\centering
\caption{Proof of \Cref{thm:siamese}, Step 4: the sixteen vertices $q_1,\dots,q_{16}$ of the level polygons $\Lambda_{+2},\Lambda_{+1},\Lambda_0,\Lambda_{-1},\Lambda_{-2}$ (groups of rows) and the slacks of the inequalities $(1)$--$(8)$ at them, as polynomials in $m$; zeros in bold. All non-zero entries are positive for $m\ge2$. The points $q_1,\dots,q_{12}$ are the images of the forced vertices; $q_{13},\dots,q_{16}$ are not vertices of the gem.}
\label{tab:slacks}
\scriptsize
\setlength{\tabcolsep}{4pt}
\resizebox{\textwidth}{!}{\begin{tabular}{l l cccccccc}
\toprule
point & $(x,y,\zeta)$ & (1) & (2) & (3) & (4) & (5) & (6) & (7) & (8)\\
\midrule
$q_{1}$ & $(-m,m,2m+2)$ & $\mathbf{0}$ & $\mathbf{0}$ & $\mathbf{0}$ & $\mathbf{0}$ & $2(m-1)(m+1)$ & $2m(m-1)(m+2)$ & $2m(m-1)$ & $2(m-1)$\\
$q_{2}$ & $(-1,m,2m+2)$ & $\mathbf{0}$ & $\mathbf{0}$ & $m-1$ & $2(m-1)(m+1)$ & $\mathbf{0}$ & $\mathbf{0}$ & $\mathbf{0}$ & $2(m-1)$\\
$q_{3}$ & $(-m,1,2m+2)$ & $\mathbf{0}$ & $m-1$ & $\mathbf{0}$ & $(m-1)(2m+1)$ & $\mathbf{0}$ & $(m-1)(2m+1)$ & $(m-1)(2m-1)$ & $\mathbf{0}$\\
\addlinespace[2pt]
$q_{4}$ & $(-m,0,1)$ & $2m+1$ & $m$ & $\mathbf{0}$ & $\mathbf{0}$ & $4m^{2}+2m-1$ & $2m(2m^{2}+2m-1)$ & $2m(2m-1)$ & $2m-1$\\
$q_{13}$ & $(-m,1-m,1)$ & $2m+1$ & $2m-1$ & $\mathbf{0}$ & $(m-1)(2m+1)$ & $2m^{2}+2m+1$ & $(m+1)(2m^{2}+2m-1)$ & $4m^{2}-3m+1$ & $1$\\
$q_{14}$ & $(-1,m-1,1)$ & $2m+1$ & $1$ & $m-1$ & $m-1$ & $4m^{2}+2m-1$ & $(m+1)(4m^{2}-2m+1)$ & $(m+1)(2m-1)$ & $4m-3$\\
\addlinespace[2pt]
$q_{5}$ & $(-m,-m,0)$ & $2(m+1)$ & $2m$ & $\mathbf{0}$ & $2m^{2}$ & $2m(m+1)$ & $2m^{2}(m+2)$ & $2m(2m-1)$ & $\mathbf{0}$\\
$q_{6}$ & $(0,-m,0)$ & $2(m+1)$ & $2m$ & $m$ & $2m(2m+1)$ & $\mathbf{0}$ & $\mathbf{0}$ & $2m(m-1)$ & $\mathbf{0}$\\
$q_{7}$ & $(0,m,0)$ & $2(m+1)$ & $\mathbf{0}$ & $m$ & $\mathbf{0}$ & $4m(m+1)$ & $2m(2m^{2}+2m-1)$ & $2m^{2}$ & $4m$\\
$q_{8}$ & $(m,m,0)$ & $2(m+1)$ & $\mathbf{0}$ & $2m$ & $2m(m+1)$ & $2m(m+1)$ & $2m(m-1)(m+1)$ & $\mathbf{0}$ & $4m$\\
\addlinespace[2pt]
$q_{9}$ & $(m,0,-1)$ & $2m+3$ & $m$ & $2m$ & $2m(2m+1)$ & $2m+1$ & $\mathbf{0}$ & $\mathbf{0}$ & $2m+1$\\
$q_{15}$ & $(m,m-1,-1)$ & $2m+3$ & $1$ & $2m$ & $(m+1)(2m+1)$ & $2m^{2}+2m-1$ & $(m-1)(2m^{2}+2m-1)$ & $m-1$ & $4m-1$\\
$q_{16}$ & $(1,1-m,-1)$ & $2m+3$ & $2m-1$ & $m+1$ & $4m^{2}+m+1$ & $2m+1$ & $(m-1)(2m+1)$ & $(m-1)(2m-1)$ & $3$\\
\addlinespace[2pt]
$q_{10}$ & $(m,-m,-2m-2)$ & $4(m+1)$ & $2m$ & $2m$ & $2m(2m+1)$ & $2(m+1)^{2}$ & $2m(m^{2}+m+1)$ & $2m^{2}$ & $2(m+1)$\\
$q_{11}$ & $(m,-1,-2m-2)$ & $4(m+1)$ & $m+1$ & $2m$ & $(m+1)(2m+1)$ & $4m(m+1)$ & $(m+1)(4m^{2}-2m+1)$ & $(m+1)(2m-1)$ & $4m$\\
$q_{12}$ & $(1,-m,-2m-2)$ & $4(m+1)$ & $2m$ & $m+1$ & $2(m^{2}+m+1)$ & $4m(m+1)$ & $2m(2m^{2}+2m-1)$ & $2m(2m-1)$ & $2(m+1)$\\
\bottomrule
\end{tabular}
}
\end{table}

All entries of \Cref{tab:slacks} and the closure of the facet list are verified symbolically in exact polynomial arithmetic (Appendix~\ref{sec:appendix}), and as an independent cross-check the theorem agrees with certified exact hull computations for $n=3,5,\dots,41$.
For $m=1$ several entries vanish and the four triangles degenerate.

\begin{figure}[htbp]
\centering
\includegraphics[width=0.72\textwidth]{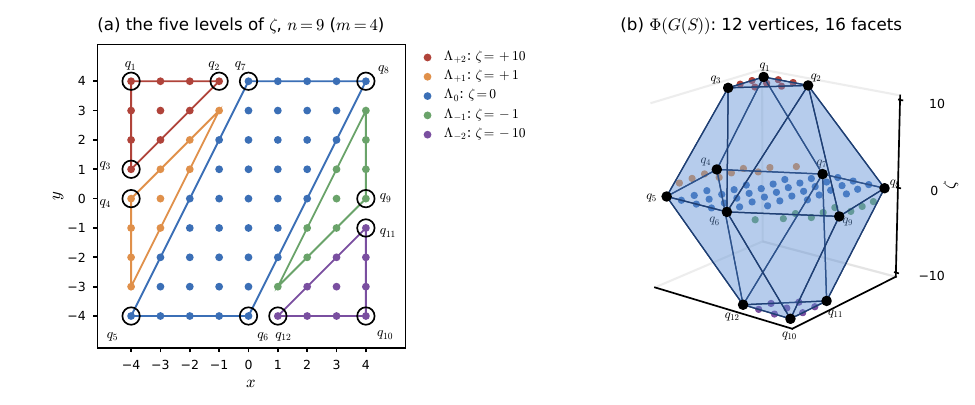}
\caption{The Siamese gem of order $n=9$ after the shear $\Phi$. (a)~The $81$ cells coloured by the level value $\zeta\in\{\pm10,\pm1,0\}$, with the five level polygons outlined and the twelve vertex cells $q_1,\dots,q_{12}$ circled. (b)~$\Phi(\mathcal G(S))$ with its $12$ vertices and $16$ facets. The combinatorics of the level polygons does not depend on $m\ge2$.}
\label{fig:siamese-layers}
\end{figure}

\subsection{The digit-exchanged square}

The method of the proof applies to other linear squares; for the digit-exchanged square it gives the following.

\begin{proposition}[Digit-exchanged gems]
\label{prop:exchanged}
Let $S'$ be the digit-exchanged square $s'$ of \Cref{cor:two-linear}, of odd order $n=2m+1\ge5$.
The vertices of $\mathcal G(S')$ are the twelve points over the cells
\[
\begin{array}{llll}
(0,0)\colon n+m+2, & (0,m-1)\colon N-n, & (0,m)\colon 1, & (0,2m)\colon N-n+m+1,\\
(1,2m)\colon m+2, & (m-1,0)\colon n(m+1), & (m+1,2m)\colon nm+1, & (2m-1,0)\colon N-n+m,\\
(2m,0)\colon m+1, & (2m,m)\colon N, & (2m,m+1)\colon n+1, & (2m,2m)\colon N-2n+m,
\end{array}
\]
listed with their entries; eight of them are forced vertices, and those over $(0,m-1)$, $(m-1,0)$, $(m+1,2m)$, $(2m,m+1)$ are not.
The $f$-vector is $(12,24,14)$, with $8$ triangular and $6$ quadrilateral facets, and $\operatorname{Iso}(\mathcal G(S'))=\{\pm\mathrm{id}\}$.
\end{proposition}

\begin{proof}
Since $\langle 2m-t\rangle=n-1-\langle t\rangle$ and $(x,y)\mapsto(-x,-y)$ maps $\ell_1$, $\ell_2$ to $2m-\ell_1$, $2m-\ell_2$, the square $S'$ is associative.
In the coordinates of Step~3, $s'=n\langle\ell_2\rangle+\langle\ell_1\rangle+1$, so $z=\lambda'(x,y)+n\zeta'$ with the affine function $\lambda'=n\ell_2+\ell_1+1-\bar s$ and $\zeta'=-n\lfloor\ell_2/n\rfloor-\lfloor\ell_1/n\rfloor$.
The floors are those of Step~3, so $\zeta'$ is constant on the same level polygons, with the values $n+1$, $n$, $0$, $-n$, $-(n+1)$ on $\Lambda_{+2},\dots,\Lambda_{-2}$.
Hence the shear $\Phi'(x,y,z)=(x,y,(z-\lambda')/n)$ maps $\mathcal G(S')$ onto the convex hull of the sixteen points $q_k$ of \Cref{tab:slacks} with the third coordinates $\pm1$ replaced by $\pm n$.
Let $H'$ be defined by the seven inequalities
\begin{align*}
&(1)\ \zeta'\le 2m+2,\qquad (2)\ y\le m,\qquad (3)\ -x\le m,\qquad (6)\ 2x-y+\zeta'\le m,\\
&(4)\ -2(m{+}1)x+(m{+}1)y-m\zeta'\le m(m{+}1),\qquad (5)\ x-y+m\zeta'\le(m{+}1)(2m{-}1),\\
&(7)\ -(2m{+}1)y+\zeta'\le m(2m{+}1),
\end{align*}
and their images $(1')$--$(7')$ under $-\mathrm{id}$.
\Cref{tab:slacks-exchanged} (Appendix~\ref{sec:appendix}) lists the slacks of $(1)$--$(7)$ at the sixteen points.
Every non-zero entry is a product of positive constants, of the linear factors $m-1$, $m$, $m+1$, $m+2$, $2m-1$, $2m+1$, $4m+3$ and of the quadratics $2m^2-1$, $2m^2+2m-1$, $2m^2+m+1$, $4m^2+m-1$, each of which is positive at $m=1$ and increasing for $m\ge1$; so all are positive for $m\ge2$, and the inequalities hold at all points as in Step~4.
The zeros of the table give the facets
\[
\begin{array}{llll}
(1)\colon q_1q_2q_3, & (2)\colon q_1q_2q_8q_7, & (3)\colon q_1q_3q_{13}q_5, & (4)\colon q_1q_5q_7,\\
(5)\colon q_2q_3q_{13}, & (6)\colon q_2q_{13}q_6q_8, & (7)\colon q_{13}q_5q_6 &
\end{array}
\]
and their antipodes ($q_{15}=-q_{13}$); the points $q_4$ and $q_{14}$ lie on the planes $(3)$ and $(6)$ but are not vertices.
The quadrilaterals are convex in the listed order: $(2)$ is the trapezoid of Step~4; on the plane $(3)$, in the coordinates $(y,\zeta')$, the four turns $(b-a)\times(c-b)$ of consecutive vertices of $q_1q_3q_{13}q_5$ are $m-1$, $(m+1)(2m-1)$, $2(2m^2-1)$, $2(m^2-1)$; and $(6)$ projects onto the quadrilateral $(-1,m)$, $(-m,1-m)$, $(0,-m)$, $(m,m)$ of the grid, with turns $2m^2-1$, $m(2m+1)$, $2m(m+1)$, $(m+1)(2m-1)$.
These fourteen polygons have $24$ edges, each shared by exactly two of them, so by \Cref{lem:closing}, with $W$ the twelve points $\pm q_1,\pm q_2,\pm q_3,\pm q_5,\pm q_6,\pm q_{13}$ ($q_7=-q_6$, $q_8=-q_5$), they are the facets of $\Phi'(\mathcal G(S'))=H'$.
Mapping back by $(i,j)=(m-y,x+m)$ gives the listed cells; the forced ones are the corners, $(0,m)$ and $(2m,m)$ (the entries $1$ and $N$), $(1,2m)$ (the smallest entry of the right column) and $(2m-1,0)$ (the largest of the left column).
For the isometries argue as in Step~5: the squared norms of the six antipodal pairs of vertices in the original coordinates are $2m^2+1$, $4m^4-4m^3-5m^2+4m+4$, $4m^4-3m^2+2$, $4m^4+4m^3-m^2-4m+2$, $m^2(4m^2+4m+3)$, $m^2(4m^2+8m+5)$ (over $(m-1,0)$, $(0,0)$, $(0,m-1)$, $(2m-1,0)$, $(2m,0)$, $(2m,m)$), pairwise distinct for $m\ge2$, and the inner products of the vertex over $(0,0)$ with the other five are non-zero for $m\ge2$.
\end{proof}

\begin{remark}[Vertex counts of linear squares]
\label{rem:linear}
The gems of other linear squares do not share the twelve-vertex count.
Over all linear magic squares, counted up to $D_4$ (each family is closed under complementation; complements are not identified), the vertex counts range over $[11,18]$ for $n=5$ ($184$ classes), $[12,25]$ for $n=7$ ($3159$ classes), $[12,31]$ for $n=9$ ($441$ classes) and $[12,36]$ for $n=11$ ($89{,}125$ classes) by exact hull computation, and over $[12,41]$, $[12,38]$, $[12,48]$, $[12,51]$ and $[12,48]$ for $n=13,15,17,19,21$ ($283{,}176$, $1472$, $1{,}713{,}664$, $3{,}558{,}249$ and $25{,}272$ classes).
For $n\ge13$ the counts are floating-point (Quickhull) counts, and every class with at most $13$ vertices or with the largest count is certified with the exact hull; for $n\le11$ the Quickhull vertex sets agree with the exact ones for every class.
For $n=5$ the minimum $11$ is attained by two classes, complements of each other, both pandiagonal and not associative, with $f$-vector $(11,23,14)$; one is $s=5\langle4i+3j\rangle+\langle4i+2j\rangle+1$ ($\mathbf u=(-1,-2)$, $\mathbf w=(-1,2)$, $\alpha=\beta=0$), the first square below.
\[
\begin{pmatrix}1&18&10&22&14\\25&12&4&16&8\\19&6&23&15&2\\13&5&17&9&21\\7&24&11&3&20\end{pmatrix}
\qquad\qquad
\begin{pmatrix}1&19&15&13&17\\23&5&14&20&3\\10&21&12&4&18\\24&9&8&22&2\\7&11&16&6&25\end{pmatrix}
\]
For $7\le n\le21$ the minimum $12$ is attained by exactly two classes, those of the Siamese and the digit-exchanged square (\Cref{thm:siamese}, \Cref{prop:exchanged}); for $n\ge13$ this relies on the floating-point counts of the other classes.
The Siamese gem is not vertex-minimal among all magic squares at order five: simulated annealing (Appendix~\ref{sec:vertexstats}) found $15$ $D_4$-classes of magic squares of order $5$ whose gems have $10$ vertices. For example, the second square above has $f$-vector $(10,22,14)$, volume $845/3$ and $8$ interior points.
\end{remark}

\subsection{Fixed steps and the largest vertex count}
\label{sec:fixed-steps}

The Siamese and the digit-exchanged gems have the same type for every odd $n\ge5$.
Every family of linear squares with fixed steps and offsets affine in $n$ has an eventually periodic $f$-vector, while the largest vertex count over all linear magic squares of order $n$ is unbounded.

Fix $\mathbf u,\mathbf w\in\Z^2\setminus\{0\}$ and rationals $\alpha_0,\alpha_1,\beta_0,\beta_1$, put $\alpha(n)=\alpha_0+\alpha_1n$, $\beta(n)=\beta_0+\beta_1n$ and $\mathcal N=\{n\ge2:\alpha(n),\beta(n)\in\Z\}$, and for $n\in\mathcal N$ let $S_n$ be the array
\[
s(c)=n\langle\mathbf u\cdot c+\alpha(n)\rangle+\langle\mathbf w\cdot c+\beta(n)\rangle+1,\qquad c=(i,j)\in\{0,\dots,n-1\}^2,
\]
and $\mathcal G_n$ the convex hull of its lifted points $(x_j,y_i,s(c)-\frac{n^2+1}2)$, the gem of $S_n$ when $S_n$ is an arrangement.
Let $\delta\ge1$ be the least integer with $\delta\alpha_1,\delta\beta_1\in\Z$, and $L$ the least common multiple of the non-zero numbers among $|u_1|$, $|u_2|$, $|w_1|$, $|w_2|$, $|u_1w_2-u_2w_1|$.
The Siamese squares form the family $\mathbf u=(1,1)$, $\mathbf w=(1,2)$, $\alpha(n)=-m=\frac{1-n}2$, $\beta(n)=-2m=1-n$, with $\mathcal N$ the odd numbers, $\delta=2$ and $L=2$.

\begin{proposition}[Fixed steps]
\label{prop:fixed-steps}
\begin{enumerate}
\item[(a)] The number of vertices of $\mathcal G_n$ is bounded on $\mathcal N$.
\item[(b)] There is an $n_0$ such that for every $n\in\mathcal N$ with $n\ge n_0$ and $\dim\mathcal G_n=3$ also $n+\delta L\in\mathcal N$, $\dim\mathcal G_{n+\delta L}=3$ and $f(\mathcal G_{n+\delta L})=f(\mathcal G_n)$.
\item[(c)] If $S_n$ is a semi-magic square, then $\dim\mathcal G_n=3$.
\end{enumerate}
\end{proposition}

\begin{proof}[Proof of (c), and outline of (a) and (b)]
If all the lifted points lay in a plane, it would not be vertical, as the grid points are not collinear; so $z=Ax+By+C$, equal row sums give $B=0$, equal column sums $A=0$, and then $z$ would be constant.
For (a) and (b), the proof of \Cref{thm:siamese}, Step~3, generalises: since $\langle t\rangle=t-n\lfloor t/n\rfloor$, one affine bijection of $\R^3$, the same for all cells, maps $\mathcal G_n$ onto the convex hull of the points $(c,\zeta(c))$ with $\zeta(c)=-n\lfloor t_A(c)/n\rfloor-\lfloor t_B(c)/n\rfloor$, where $t_A(c)=\mathbf u\cdot c+\alpha(n)$ and $t_B(c)=\mathbf w\cdot c+\beta(n)$.
The pairs of floors range over a finite set that does not depend on $n$, and each level set is the set of lattice points of a polygon $\{x\in\R^2:Mx\le ng+h\}$ with $M$, $g$, $h$ independent of $n$.
For large $n$ in a fixed residue class modulo $\delta L$, the vertices of the hull of the lattice points of such a polygon come from a fixed finite list of points $np+q$ (\Cref{lem:linear-polygons}), and the face lattice of the hull of finitely many points that move affinely with $n$ is eventually constant (\Cref{lem:moving-points}).
The details are in Appendix~\ref{sec:app-fixed-steps}.
\end{proof}

So the vertex count along such a family is bounded, as it is for the Siamese and the digit-exchanged squares, but the bound, $n_0$ and the eventual $f$-vectors are not explicit.
The least period can be smaller than $\delta L$: for the Siamese family $\delta L=4$ and the $f$-vector is constant for $n\ge5$.
In computations it can equal $\delta L$: for $\mathbf u=(1,3)$, $\mathbf w=(3,1)$, $\alpha=\beta=0$ we have $\delta=1$ and $L=24$, and the exact hulls satisfy $f(\mathcal G_n)=f(\mathcal G_{n+24})$ for $150\le n\le197$, with no smaller shift doing so on that range.

The bound depends on the family, and the largest vertex counts in \Cref{rem:linear} grow with $n$.
That they are unbounded follows from Rubin's example of triangles whose integer hulls have arbitrarily many vertices \cite{rubin1970}.
Zolotykh's survey \cite[Sect.~3]{zolotykh2006vertices} describes Rubin's triangles by Fibonacci knapsack inequalities with $k+3$ vertices; its printed inequality, with $F_{2k-1}$ in place of $F_{2k+1}$, has fewer vertices, and we have not seen Rubin's paper, so we state and prove the version we use (Appendix~\ref{sec:app-fixed-steps}).
Here $F_t$ are the Fibonacci numbers, extended to all $t\in\Z$ by $F_0=0$, $F_1=1$, $F_{t+1}=F_t+F_{t-1}$.

\begin{lemma}[Rubin's triangle]
\label{lem:fib-triangle}
Let $k\ge1$, $a=F_{2k}$, $b=F_{2k+1}$, $R=b^2-1$ and $T_k=\{(x,y)\in\R^2:x\ge0,\ y\ge0,\ ax+by\le R\}$.
The convex hull of $T_k\cap\Z^2$ has exactly $k+3$ vertices, $(0,0)$ and $(F_{2t},\,b-F_{2t-1})$ for $0\le t\le k+1$.
\end{lemma}

\begin{proposition}[Many vertices]
\label{prop:many-vertices}
Let $k$, $a$, $b$, $R$ be as in \Cref{lem:fib-triangle}, and let $n$ be a prime with $n>\max(b^2,3b)$.
Then $s(i,j)=n\langle ai+bj+n-1-R\rangle+\langle i+3j\rangle+1$ is a magic square of order $n$, and the points of its gem over the $k+3$ cells $(i,j)$ that are vertices of $\operatorname{conv}(T_k\cap\Z^2)$ are vertices of the gem.
\end{proposition}

\begin{proof}
As $n>3b\ge6$ is prime, every integer $\sigma$ with $0<|\sigma|<n$ is a unit of $\Z_n$.
The numbers $a$, $b$, $1$, $3$ and $3a-b=F_{2k}+F_{2k-2}$, and the diagonal steps $a+b=F_{2k+2}$, $a-b=-F_{2k-1}$, $4$ and $-2$, are of this kind, since $b^2<n$ and $3b<n$; so (a) of \Cref{prop:linear-magic} holds and (b) is vacuous, and the square is magic.
Put $t_A=ai+bj+\alpha$ with $\alpha=n-1-R$, which lies in $\{1,\dots,n-1\}$ as $n\ge b^2+1$, and $t_B=i+3j$.
As in Step~3 of the proof of \Cref{thm:siamese}, $s-\frac{n^2+1}2=\lambda+n\zeta$ with $\lambda=nt_A+t_B+1-\frac{n^2+1}2$ affine and $\zeta=-n\lfloor t_A/n\rfloor-\lfloor t_B/n\rfloor$, so the shear $(x,y,z)\mapsto(x,y,(z-\lambda)/n)$ maps the gem onto the convex hull of the points $(x_j,y_i,\zeta(i,j))$.
On the grid $t_A,t_B\ge0$, so $\zeta\le0$, with equality if and only if $ai+bj\le R$ and $i+3j\le n-1$.
Every lattice point $(i,j)$ of $T_k$ is a cell, as $i\le R<n-1$ and $j<b<n$, and satisfies $i+3j\le\max(R/a,3R/b)<\max(b^2,3b)<n$; so the cells with $\zeta=0$ are exactly the lattice points of $T_k$.
If $Q\subset\R^3$ is finite and $Q_{\rm top}$ is the set of its points of largest third coordinate, every vertex $q$ of $\operatorname{conv}Q_{\rm top}$ is a vertex of $\operatorname{conv}Q$: if a linear function $\ell$ is maximal on $Q_{\rm top}$ only at $q$, then $\zeta+\varepsilon\ell$ is maximal on $Q$ only at $q$ for small $\varepsilon>0$.
Applied to the sheared points, this proves the claim.
\end{proof}

\begin{corollary}
\label{cor:many-vertices}
For every prime $n\ge7$ some linear magic square of order $n$ has a gem with at least $\lceil\ln n/(4\ln\varphi)\rceil+2\ge0.519\ln n+2$ vertices, where $\varphi=(1+\sqrt5)/2$.
In particular the largest vertex count of a linear magic square of prime order $n$ is unbounded.
\end{corollary}

\begin{proof}
Put $k=\lceil\ln n/(4\ln\varphi)\rceil-1$. Then $k\ge1$, as $n\ge7>\varphi^4$, and $\varphi^{4k}<n$.
Since $F_t\le\varphi^{t-1}$ for $t\ge1$, $b^2=F_{2k+1}^2\le\varphi^{4k}<n$; and $3b<n$, as $3b=6<n$ for $k=1$ and $3b<b^2$ for $k\ge2$.
Now apply \Cref{prop:many-vertices}.
\end{proof}

The bound is far below the computed maxima of \Cref{rem:linear} (the square of the proof has $18$ vertices at $n=7$, where the bound is $4$) and covers only prime orders, but its order of magnitude $\Theta(\ln n)$ is sharp: choosing balanced residues $u_1,u_2,w_1,w_2\in(-n/2,n/2]\cap\Z$ and $\alpha,\beta\in\{0,\dots,n-1\}$ restricts the floor pair $(\lfloor t_A/n\rfloor,\lfloor t_B/n\rfloor)$ in Step~1 of Appendix~\ref{sec:app-fixed-steps} to at most $25$ values on the grid, so the sheared points lie on at most $25$ level polygons $P_{ab}(n)\subseteq[0,n-1]^2$ defined by at most $8$ linear inequalities each, and the integer hull of any polygon with a bounded number of sides in $[0,n-1]^2$ has $O(\ln n)$ vertices \cite[Sect.~2]{zolotykh2006vertices}.

\section{Open questions}
\label{sec:discussion}

\begin{enumerate}\setlength\itemsep{2pt}
\item \emph{Hull versus point set.} Is $\operatorname{Iso}(\mathcal G(S))=\operatorname{Iso}(P_S)$ for every magic, or semi-magic, square $S$, or more generally for every arrangement whose heights have vanishing first moments? By \Cref{prop:pointset} this would determine the isometry groups of all these gems; it holds for order four (\Cref{prop:iso4}) and for the Siamese and digit-exchanged squares (\Cref{thm:siamese}, \Cref{prop:exchanged}). For odd $n$ the self-complementary squares are associative (\Cref{lem:selfcomp}), so the point set of an order-five square has isometry group $\{\pm\mathrm{id}\}$ or trivial; whether the same holds for the gems is open, and an exhaustive check over the $275{,}305{,}224$ classes of order five would take about seven CPU-weeks at the rate of \Cref{sec:census-4}.
\item \emph{Growth of the vertex count for arbitrary arrangements.} For uniformly random arrangements the mean vertex count grows from $8.22$ at $n=3$ to $27.91$ at $n=10$, and the ratio $\mathbb E[V]/\ln^2N$ decreases slowly from $1.70$ to $1.32$ (\Cref{tab:vertices}; estimates). For $N$ independent uniform points in a $3$-polytope the expected vertex count is $\Theta(\ln^2N)$ \cite{barany1993random}; our points are not independent (the grid is fixed and only the heights are permuted), so this is only a heuristic analogy. Is $\mathbb E[V]=\Theta(\ln^2 N)$ for random arrangements? By \Cref{sec:magic-vs-random} the line conditions hardly change the mean vertex count at orders four and five; do they change its growth rate?
\item \emph{Secondary cones.} By \Cref{prop:subdivisions} the face lattice of the gem of a height function $z$ is fixed by the secondary cones containing $z$ and $-z$, so the combinatorics of the gems of a family of arrangements is the set of pairs of chambers of the secondary fan of the grid that it meets. Which pairs are met by the permuted heights that satisfy the line conditions?
\item \emph{Linear squares.} By \Cref{prop:fixed-steps} the vertex count is bounded along every family of linear squares with fixed steps and offsets affine in $n$, and the $f$-vector is eventually periodic, but neither the bound nor the onset of periodicity is explicit. Is the vertex count bounded uniformly in the offsets, for fixed steps? Over all linear magic squares of prime order $n$ the largest vertex count is of order $\Theta(\ln n)$ (\Cref{cor:many-vertices}), with the lower bound $0.519\ln n+2$ far below the computed values ($25$, $36$, $41$, $48$, $51$ for $n=7,11,13,17,19$; \Cref{rem:linear}). What is the asymptotic constant, and what happens for composite $n$?
\item \emph{Inner core gems and solid mass balance.} While the $N$ points always have centroid $\mathbf 0$, the uniform solid $\mathcal G(S)$ does so for only $51$ of the $880$ order-four gems (Appendix~\ref{sec:further4}), because the outer hull ignores the $N-V$ non-vertex points. Peeling $\operatorname{vert}\mathcal G(S)$ leaves the inner core gem $\mathcal G_1(S)=\operatorname{conv}(P_S\setminus\operatorname{vert}\mathcal G(S))$, whose point set has centroid $-V/(N-V)$ times the vertex centroid of $\mathcal G(S)$. For order four $\mathcal G_1(S)$ is $3$-dimensional for $511$ of the $880$ gems ($207$ of which already have solid centroid $\mathbf 0$, and $33$ of the half-turn gems have outer and inner solid centroids on opposite rays from $\mathbf 0$, allowing a two-density composite solid to balance at $\mathbf 0$). How many convex layers arise for larger orders?
\end{enumerate}

\section*{Acknowledgments}
\label{sec:ai_statement}

Large language models (OpenAI's ChatGPT, Anthropic's Claude, and Google's Gemini) were used to assist with mathematical exploration, software implementation, and manuscript editing; all proofs and exact computational certificates have been verified by the author, who takes full responsibility for the content. This version substantially revises earlier arXiv drafts by separating the polynomial-moment and permutation-energy material into a companion paper, adding the exact order-four census and the theorems on Siamese and linear magic squares, and correcting earlier numerical and terminological inaccuracies.

\bibliographystyle{plain}
{\small\setlength{\itemsep}{0pt}\bibliography{references}}

\appendix
\section{Computational methods}
\label{sec:appendix}

All computations use Python 3.9 with NumPy 2.0, SciPy 1.13 (Quickhull), SymPy 1.13 and NetworkX 3.2; the scripts and data are the ancillary files in \texttt{anc/}, whose \texttt{README} lists the commands.

\paragraph{Exact hulls.}
After doubling, all coordinates are integers.
A floating-point Quickhull triangulation only proposes candidate facet planes, recomputed as primitive integer equations through three points each; the hull is then certified in integer arithmetic:
(i) the planes are pairwise distinct;
(ii) every point satisfies every facet inequality;
(iii) the points on each plane have an exact $2$-dimensional convex hull (monotone chain with integer cross products after an injective projection), the facet polygon;
(iv) every edge of every facet polygon lies in exactly two facets.
By \Cref{lem:closing}, (i)--(iv) certify the hull; strict convexity across every edge and $V-E+F=2$ are checked as redundant tests.
An independent implementation that enumerates the supporting planes through all triples of points, without Quickhull, gives the same vertex sets, $f$-vectors and volumes for all $880$ order-four gems, the Siamese squares of orders $3$ to $21$ and the order-five squares of \Cref{rem:linear}. For the Siamese family the script that generates \Cref{tab:slacks} checks these conditions symbolically for every $m\ge2$.
This certificate is applied to every exact hull in the paper (order four, Siamese, digit-exchanged and linear squares, the $15$ ten-vertex order-five squares, and the samples of \Cref{sec:magic-vs-random}).
Vertices are the points on at least three facet planes, interior points are those satisfying every facet inequality strictly, and volumes are exact rationals.
For the $880$ order-four gems this takes about $15$\,s in total.

\paragraph{Combinatorial types, isometries, congruence.}
Combinatorial types are the isomorphism classes of the bipartite vertex--facet incidence graphs, under isomorphisms that map vertices to vertices and facets to facets (so a polytope and its dual are not identified); candidates are grouped by a Weisfeiler--Lehman hash and confirmed by explicit isomorphism tests.
Isometry groups of a finite set $X$ (a vertex set or a full point set) are computed exactly: $X$ is centred at its centroid, every isometry is then orthogonal and determined by the images of three linearly independent points, which are searched among the points with the same squared norms and mutual inner products, and a candidate is accepted only if the induced permutation of $X$ preserves all pairwise integer squared distances.
Congruence of two finite sets is tested by comparing the sorted multisets of squared distances and then by an isomorphism test of the complete graphs with squared distances as edge labels; for the $880$ gems the multisets alone certify the classes (\Cref{prop:congruence}).
\Cref{prop:pointset} is checked in the same way for all $880$ squares and for the Siamese squares of orders $3$ to $15$.

\paragraph{Polynomial identities and ancillary scripts.}
The slack polynomials of \Cref{tab:slacks} are computed with SymPy (\texttt{anc/siamese\_slacks.py}); a polynomial $p$ is certified positive for all integers $m\ge2$ by non-negativity of all coefficients of $p(2+t)$ and $p(2)>0$.
The same script checks the closure of the facet list and the norms and inner products of Step~5.
For the digit-exchanged square, \texttt{anc/digit\_exchanged\_slacks.py} applies the same certificate, written once for both squares (it reproduces \Cref{thm:siamese} as a cross-check), to the slacks of \Cref{tab:slacks-exchanged}; it also checks, for $n=5,\dots,61$, that the sheared points of $s'$ have the five level sets used in the proof of \Cref{prop:exchanged} and satisfy all fourteen inequalities.
Among the other ancillary files in \texttt{anc/}, \texttt{squares880.txt} and \texttt{all7040\_dudeney.tsv} list the Fr\'enicle standard forms and Dudeney group labels of the order-four squares; \texttt{verify\_structure.py}, \texttt{congruent\_facets.py} and \texttt{vertex\_determination.py} verify the order-four census, facet congruences and vertex-set determination (Sections~\ref{sec:symmetry} and~\ref{sec:census}); \texttt{energy\_zeros\_count.py} and \texttt{magic\_vs\_random.py} produce \Cref{tab:magic-vs-random}; and \texttt{linear\_magic\_count.py}, \texttt{linear\_squares.py}, \texttt{linear\_census.py}, \texttt{linear\_fixed\_steps.py} and \texttt{sa\_n5\_V10.txt} verify the counts and bounds of \Cref{sec:siamese-gems}.

\begin{table}[htbp]
\centering
\caption{Proof of \Cref{prop:exchanged}: the sixteen vertices $q_1,\dots,q_{16}$ of the level polygons of the digit-exchanged square after the shear $\Phi'$ (heights $n+1$, $n$, $0$, $-n$, $-(n+1)$) and the slacks of the inequalities $(1)$--$(7)$ at them, as polynomials in $m$; zeros in bold. All non-zero entries are positive for $m\ge2$. The points $q_4,q_{14},q_9,q_{16}$ are not vertices of the gem.}
\label{tab:slacks-exchanged}
\scriptsize
\setlength{\tabcolsep}{4pt}
\resizebox{0.9\textwidth}{!}{\begin{tabular}{l l ccccccc}
\toprule
point & $(x,y,\zeta^\prime)$ & (1) & (2) & (3) & (4) & (5) & (6) & (7)\\
\midrule
$q_{1}$ & $(-m,m,2m+2)$ & $\mathbf{0}$ & $\mathbf{0}$ & $\mathbf{0}$ & $\mathbf{0}$ & $m-1$ & $2(m-1)$ & $2(2m^{2}-1)$\\
$q_{2}$ & $(-1,m,2m+2)$ & $\mathbf{0}$ & $\mathbf{0}$ & $m-1$ & $2(m-1)(m+1)$ & $\mathbf{0}$ & $\mathbf{0}$ & $2(2m^{2}-1)$\\
$q_{3}$ & $(-m,1,2m+2)$ & $\mathbf{0}$ & $m-1$ & $\mathbf{0}$ & $(m-1)(m+1)$ & $\mathbf{0}$ & $m-1$ & $(m+1)(2m-1)$\\
\addlinespace[2pt]
$q_{4}$ & $(-m,0,2m+1)$ & $1$ & $m$ & $\mathbf{0}$ & $m^{2}$ & $m-1$ & $m-1$ & $(m-1)(2m+1)$\\
$q_{13}$ & $(-m,1-m,2m+1)$ & $1$ & $2m-1$ & $\mathbf{0}$ & $2m^{2}-1$ & $\mathbf{0}$ & $\mathbf{0}$ & $\mathbf{0}$\\
$q_{14}$ & $(-1,m-1,2m+1)$ & $1$ & $1$ & $m-1$ & $2m^{2}-1$ & $m-1$ & $\mathbf{0}$ & $2(m-1)(2m+1)$\\
\addlinespace[2pt]
$q_{5}$ & $(-m,-m,0)$ & $2(m+1)$ & $2m$ & $\mathbf{0}$ & $\mathbf{0}$ & $(m+1)(2m-1)$ & $2m$ & $\mathbf{0}$\\
$q_{6}$ & $(0,-m,0)$ & $2(m+1)$ & $2m$ & $m$ & $2m(m+1)$ & $2m^{2}-1$ & $\mathbf{0}$ & $\mathbf{0}$\\
$q_{7}$ & $(0,m,0)$ & $2(m+1)$ & $\mathbf{0}$ & $m$ & $\mathbf{0}$ & $2m^{2}+2m-1$ & $2m$ & $2m(2m+1)$\\
$q_{8}$ & $(m,m,0)$ & $2(m+1)$ & $\mathbf{0}$ & $2m$ & $2m(m+1)$ & $(m+1)(2m-1)$ & $\mathbf{0}$ & $2m(2m+1)$\\
\addlinespace[2pt]
$q_{9}$ & $(m,0,-2m-1)$ & $4m+3$ & $m$ & $2m$ & $m(m+2)$ & $4m^{2}+m-1$ & $m+1$ & $(m+1)(2m+1)$\\
$q_{15}$ & $(m,m-1,-2m-1)$ & $4m+3$ & $1$ & $2m$ & $2m+1$ & $2(m+1)(2m-1)$ & $2m$ & $2m(2m+1)$\\
$q_{16}$ & $(1,1-m,-2m-1)$ & $4m+3$ & $2m-1$ & $m+1$ & $2m+1$ & $4m^{2}+m-1$ & $2m$ & $2(2m+1)$\\
\addlinespace[2pt]
$q_{10}$ & $(m,-m,-2m-2)$ & $4(m+1)$ & $2m$ & $2m$ & $2m(m+1)$ & $4m^{2}+m-1$ & $2$ & $2(m+1)$\\
$q_{11}$ & $(m,-1,-2m-2)$ & $4(m+1)$ & $m+1$ & $2m$ & $(m+1)^{2}$ & $2(m+1)(2m-1)$ & $m+1$ & $2m^{2}+m+1$\\
$q_{12}$ & $(1,-m,-2m-2)$ & $4(m+1)$ & $2m$ & $m+1$ & $2(m+1)$ & $2(m+1)(2m-1)$ & $2m$ & $2(m+1)$\\
\bottomrule
\end{tabular}
}
\end{table}

\section{Lattice-point counts and vertex statistics}
\label{sec:lattice-stats}

\begin{table}[htbp]
\centering
\caption{Mean number of vertices of the gem. Top: exact values (all $8$ and $7040$ magic squares of orders $3$ and $4$; the Siamese squares). Bottom: estimates from samples. Random: uniformly random arrangements, $10^5$ samples per order (standard errors at most $0.010$). Magic, annealed: $59{,}993$, $24{,}000$, $8{,}000$ and $2{,}400$ distinct $D_4$-classes for $n=5,\dots,8$ (standard errors $0.007$, $0.013$, $0.025$, $0.054$), from a sampler whose distribution is not known; hulls in floating point. The certified order-five samples of \Cref{sec:magic-vs-random} are separate.}
\label{tab:vertices}
\small
\begin{tabular}{l rrrrrrrr}
\toprule
$n$ & 3 & 4 & 5 & 6 & 7 & 8 & 9 & 10\\
\midrule
\multicolumn{9}{l}{\emph{exact}}\\
magic, all squares & 8 & 12.23 & -- & -- & -- & -- & -- & --\\
Siamese & 8 & -- & 12 & -- & 12 & -- & 12 & --\\
\midrule
\multicolumn{9}{l}{\emph{estimates}}\\
random arrangements & 8.22 & 12.06 & 15.39 & 18.39 & 21.08 & 23.53 & 25.79 & 27.91\\
magic, annealed & -- & -- & 15.38 & 18.41 & 21.08 & 23.56 & -- & --\\
\bottomrule
\end{tabular}
\end{table}

\paragraph{Ehrhart polynomials of the order-four gems.}
\label{sec:ehrhart}
For even $n$ all coordinates of every $\mathbf p_{ij}$ are half-odd integers, so $Q(S)=\mathcal G(S)+(\frac12,\frac12,\frac12)$ is a lattice polytope, with Ehrhart polynomial $L_S(t)=\#(tQ\cap\Z^3)$; it is the same for every lattice translate of $Q$ and for all images $gS$, $g\in D_4$.
For a lattice $3$-polytope, $L_S(t)=c_3t^3+c_2t^2+c_1t+1$, where $c_3$ is the volume and $c_2$ is half the sum of the areas of the facets, each measured relative to the lattice in its plane \cite[Ch.~3 and~5]{beck2015computing}; and $L_S(1)$ is the number of lattice points of $Q$.
So $L_S$ is determined by the volume, this surface term and the number of lattice points.
Over the $880$ gems (\texttt{anc/ehrhart\_880.py}) there are $554$ distinct polynomials, against $616$ congruence classes and $155$ volumes; so this count only measures how often non-congruent gems share these three numbers.
For the Lo Shu square (odd order: the gem itself is a lattice polytope) $L(t)=1+t+13t^2+18t^3$, and for D\"urer's square $L(t)=1-3t+44t^2+92t^3$.

\paragraph{Vertex counts for larger orders (numerical).}
\label{sec:vertexstats}
In \Cref{tab:vertices} the magic squares of orders $5$ to $8$ are found by simulated annealing on the sum of squared line excesses (\texttt{anc/sa\_magic.c}).
For random arrangements the fraction $\mathbb E[V]/N$ falls from $0.91$ ($n=3$) to $0.28$ ($n=10$).
For $n=5,\dots,8$ the annealed sample means are close to the random ones (within $0.06$ vertices), though we emphasise that this is a statement about the samples only, as the sampling distribution of the annealer for $n\ge5$ is not known (compare \cite{pinn1998number}).
At $n=4$, where all magic squares are known, $2\cdot10^4$ annealed squares hit all $880$ classes, each between $10$ and $41$ times (consistent with the uniform mean $20000/880\approx22.7$; \texttt{anc/vertex\_stats\_sa4.json}), and the annealer for semi-magic squares reproduces the uniform distribution of every invariant of \Cref{tab:magic-vs-random} (\texttt{anc/magic\_vs\_random.log}).
The smallest annealed vertex counts are $10$, $12$, $14$ and $16$ for $n=5,\dots,8$ (\Cref{rem:linear}).
The hulls of these samples are computed by Qhull in floating point, with the vertex sets of the first $100$ hulls of each batch rechecked in exact arithmetic; the seeds are $7000n+c$ for random arrangements (\texttt{anc/vertex\_statistics.py}) and $1000s+n$, $s=1,\dots,16$, for the annealer (\texttt{anc/run\_sa.sh}).
The order-five samples of \Cref{sec:magic-vs-random} were drawn separately ($2\cdot10^4$ per family, annealer seeds $9000s+5$, $s=1,\dots,8$, every hull certified exactly); their magic sample has mean vertex count $15.39$, against $15.38$ here.

\section{Further order-four data}
\label{sec:further4}

\paragraph{Solid centroid, inner core gem and mixed moment.}
The centre of mass of the solid $\mathcal G(S)$ is not the origin for $829$ of the $880$ gems; the other $51$ include the $48$ associative gems (\texttt{anc/solid\_centroid.py}).
Removing $\operatorname{vert}\mathcal G(S)$ leaves $N-V$ non-vertex points, which are non-empty for $862$ of the $880$ gems and span a $3$-dimensional inner core gem $\mathcal G_1(S)=\operatorname{conv}(P_S\setminus\operatorname{vert}\mathcal G(S))$ for $511$ of them ($2$-dimensional for $175$, $1$-dimensional for $162$ and a single point for $14$); for $207$ of the $511$ the solid centroid of $\mathcal G_1(S)$ is the origin.
The mixed moment $|\sum z\,xy|$ takes $48$ distinct values between $0$ and $60$ and vanishes for $56$ classes (tables in \Cref{fig:weighted}), among them the $48$ associative ones, for which $z$ is odd under $(x,y)\mapsto(-x,-y)$ and $xy$ is even.

\paragraph{Difference-body ratio.}
For a convex body $K\subset\R^3$ with non-empty interior let $\sigma(K)=\operatorname{vol}(K-K)/(8\operatorname{vol}K)$. By the Brunn--Minkowski inequality \cite{schneider2014convex}, $\sigma(K)\ge1$, with equality if and only if $K$ is centrally symmetric; by the Rogers--Shephard inequality \cite{rogers1957difference}, $\sigma(K)\le\binom63/8=\frac52$, with equality if and only if $K$ is a simplex; see \cite{grunbaum1963measures,schneider2014convex} for other measures of symmetry.
By \Cref{prop:iso4}, $\sigma=1$ exactly for the $48$ associative gems; for the other $832$, $\sigma\in[1057/1018,\,145/88]\approx[1.038,\,1.648]$ (exact, \texttt{anc/difference\_body\_ratio.py}).

\paragraph{Dudeney's groups.}
Each of Dudeney's groups with more than eight members has gems with at least three vertex counts and $20$ combinatorial types; the $8$-vertex gems lie in groups VI (three) and VIII (four) (\texttt{anc/verify\_numbers.py}).
The $48$ pandiagonal squares (group I \cite{ollerenshaw1998most}) have $11$ to $15$ vertices and one to four interior points, and none is associative.

\section{Proofs for Section~\ref{sec:fixed-steps}}
\label{sec:app-fixed-steps}

We use the notation of \Cref{sec:fixed-steps}; $c=(i,j)$ runs over the cells, $t_A(c)=\mathbf u\cdot c+\alpha(n)$ and $t_B(c)=\mathbf w\cdot c+\beta(n)$ are integers for $n\in\mathcal N$, and $f$ denotes the $f$-vector.

\paragraph{Step 1: levels.}
Since $\langle t\rangle=t-n\lfloor t/n\rfloor$, we have $s(c)-\frac{n^2+1}2=\lambda(c)+n\zeta(c)$ with $\lambda(c)=nt_A(c)+t_B(c)+1-\frac{n^2+1}2$, an affine function of $c$, and $\zeta(c)=-n\lfloor t_A(c)/n\rfloor-\lfloor t_B(c)/n\rfloor$.
As the cell $c$ is an affine function of $(x_j,y_i)$, the map $\Psi_n(x,y,z)=\bigl(c(x,y),(z-\lambda(c(x,y)))/n\bigr)$ is an affine bijection of $\R^3$ that maps the lifted point over $c$ to $(c,\zeta(c))$; so $\Psi_n(\mathcal G_n)$ has the same $f$-vector as $\mathcal G_n$, and the same vertex cells.
On the grid $|t_A(c)|\le(|u_1|+|u_2|)n+|\alpha_0|+|\alpha_1|n$, so the pair $(a,b)=(\lfloor t_A/n\rfloor,\lfloor t_B/n\rfloor)$ lies in a finite set $\mathcal L\subset\Z^2$ independent of $n$.
The cells with a given pair are the lattice points of
\[
P_{ab}(n)=\{x\in\R^2:\ 0\le x_1,x_2\le n-1,\ na\le\mathbf u\cdot x+\alpha(n)\le na+n-1,\ nb\le\mathbf w\cdot x+\beta(n)\le nb+n-1\},
\]
a polygon of the form $\{x:Mx\le ng^{ab}+h\}$, where the rows of $M\in\Z^{8\times2}$ are $\pm e_1$, $\pm e_2$, $\pm\mathbf u$, $\pm\mathbf w$, the vector $h\in\mathbb Q^8$ does not depend on $(a,b)$ or $n$, and $g^{ab}\in\mathbb Q^8$ does not depend on $n$ and satisfies $\delta g^{ab}\in\Z^8$.
Since $\operatorname{conv}(X\times\{\zeta\})=\operatorname{conv}(X)\times\{\zeta\}$ for $X\subset\R^2$,
\begin{equation}
\label{eq:levels}
\Psi_n(\mathcal G_n)=\operatorname{conv}\bigcup_{(a,b)\in\mathcal L}\operatorname{conv}\bigl(P_{ab}(n)\cap\Z^2\bigr)\times\{-na-b\}.
\end{equation}
Since $P_\infty^{ab}=\{x:Mx\le g^{ab}\}\subseteq[0,1]^2$ is bounded, the following lemma applies to each $P_{ab}(n)$.

\begin{lemma}[Integer hulls of a linear family of polygons]
\label{lem:linear-polygons}
Let $M\in\Z^{K\times2}$ have rows $\mathbf r_k$, let $g,h\in\mathbb Q^K$, put $P_n=\{x\in\R^2:Mx\le ng+h\}$, and suppose that $P_\infty=\{x:Mx\le g\}$ is bounded.
Let $D$ be the least common multiple of the denominators $q_v$ of the vertices $v$ of $P_\infty$ ($q_v$ is the least positive integer with $q_vv\in\Z^2$; $D=1$ if $P_\infty=\emptyset$).
There are $n_1$ and, for each $r\in\{0,\dots,D-1\}$, a finite set $Z_r\subset\mathbb Q^2\times\mathbb Q^2$ such that for all $n\ge n_1$ with $n\equiv r\pmod D$ the points $np+q$, $(p,q)\in Z_r$, lie in $P_n\cap\Z^2$, and $\operatorname{conv}(P_n\cap\Z^2)=\operatorname{conv}\{np+q:(p,q)\in Z_r\}$.
\end{lemma}

This is a planar case of a theorem of Calegari and Walker \cite[Thm.~3.5]{calegari2013integer}, which says that the vertices of the integer hulls of a rational family of polyhedra of size $O(n)$ have eventually quasi-polynomial coordinates; the short direct proof below also gives the explicit period $D$, which yields the period $\delta L$ in \Cref{prop:fixed-steps}.

\begin{proof}
If $P_\infty=\emptyset$, Farkas' lemma gives $y\ge0$ with $yM=0$ and $y\cdot g<0$; then $0=yMx\le n\,y\cdot g+y\cdot h$ fails for large $n$, so $P_n=\emptyset$, and we take $Z_r=\emptyset$.
Otherwise, for a vertex $v$ of $P_\infty$ let $I_v=\{k:\mathbf r_k\cdot v=g_k\}$, $C_v=\{y:\mathbf r_k\cdot y\le h_k\ (k\in I_v)\}$ and $T_v=\{y:\mathbf r_k\cdot y\le0\ (k\in I_v)\}$.

(i) $P_n\subseteq nv+C_v$ for all $n$, since $\mathbf r_k\cdot(x-nv)=\mathbf r_k\cdot x-ng_k$ for $k\in I_v$.

(ii) Let $\varepsilon>0$ be the least slack $g_k-\mathbf r_k\cdot v$ over all vertices $v$ and $k\notin I_v$ ($\varepsilon=\infty$ if there is no such pair), and $\nu=\max_k|\mathbf r_k|$.
If $y\in C_v$, $|y|\le\rho$, $n\ge1$ and $n\varepsilon\ge\nu\rho-h_k$ for all $k$, then $nv+y\in P_n$: for $k\notin I_v$, $\mathbf r_k\cdot(nv+y)\le n(g_k-\varepsilon)+\nu\rho\le ng_k+h_k$.

(iii) If some $C_v$ is empty, $P_n=\emptyset$ for all $n$ by (i), and we take $Z_r=\emptyset$.
Otherwise $C_v=Q_v+T_v$ with $Q_v$ a polytope and $T_v$ the cone spanned by integer vectors $y_1,\dots,y_s$ (Minkowski--Weyl).
Put $B_v=\{\sum_l\theta_ly_l:0\le\theta_l\le1\}$ and $F_v(n)=(nv+Q_v+B_v)\cap\Z^2$.
Every lattice point $x=nv+q+\sum_l\lambda_ly_l$ of $nv+C_v$ ($q\in Q_v$, $\lambda_l\ge0$) lies in $F_v(n)+T_v$: subtract $\sum_l\lfloor\lambda_l\rfloor y_l$.

(iv) If $n\equiv r'\pmod{q_v}$, then $(n-r')v\in\Z^2$, so $F_v(n)=nv+E_v(r')$ with the finite set $E_v(r')=F_v(r')-r'v\subset Q_v+B_v$.

(v) Let $\rho$ be the largest norm of a point of the sets $Q_v+B_v$, and $n_1\ge1$ with $n_1\varepsilon\ge\nu\rho-h_k$ for all $k$. By (ii), $F_v(n)\subseteq P_n\cap\Z^2$ for $n\ge n_1$.

(vi) Let $n\ge n_1$, $H_n=\operatorname{conv}(P_n\cap\Z^2)$ and $R_n=\operatorname{conv}\bigcup_vF_v(n)$, so that $R_n\subseteq H_n$ by (v); we may assume $P_n\cap\Z^2\ne\emptyset$.
For a linear function $\ell$ choose a vertex $v$ of $P_\infty$ at which $\ell$ is maximal on $P_\infty$.
For $y\in T_v$ and small $\eta>0$ the point $v+\eta y$ lies in $P_\infty$, as the rows $k\notin I_v$ have positive slack at $v$; so $\ell\le0$ on $T_v$.
By (i) and (iii), $P_n\cap\Z^2\subseteq F_v(n)+T_v$, so $\max_{H_n}\ell\le\max_{F_v(n)}\ell\le\max_{R_n}\ell$.
As this holds for every $\ell$, $H_n=R_n$.
Now take $Z_r=\{(v,e):v\text{ a vertex of }P_\infty,\ e\in E_v(r\bmod q_v)\}$.
\end{proof}

A vertex of $P_\infty^{ab}$ solves two independent equations $\mathbf r_k\cdot v=g^{ab}_k$, and the determinant of two non-parallel rows of $M$ is $\pm1$, $\pm u_1$, $\pm u_2$, $\pm w_1$, $\pm w_2$ or $\pm(u_1w_2-u_2w_1)$; by Cramer's rule and $\delta g^{ab}\in\Z^8$, its denominator divides $\delta$ times the absolute value of the determinant.
So the number $D_{ab}$ of \Cref{lem:linear-polygons} for $P_{ab}(n)$ divides $\delta L$.

\begin{lemma}[Points moving affinely]
\label{lem:moving-points}
Let $\mathcal K$ be finite, $A_\kappa,B_\kappa\in\mathbb Q^3$ and $\pi_\kappa(n)=nA_\kappa+B_\kappa$ for $\kappa\in\mathcal K$.
There is $n_2$ such that whether $\operatorname{conv}\{\pi_\kappa(n)\}$ is $3$-dimensional is the same for all $n\ge n_2$, and, if it is, so are its $f$-vector and the index sets $\{\kappa:\pi_\kappa(n)\in F\}$ of its facets $F$.
\end{lemma}

\begin{proof}
For $\kappa_1,\dots,\kappa_4\in\mathcal K$ the determinant $\Delta_{(\kappa_1,\dots,\kappa_4)}(n)$ of the $4\times4$ matrix with rows $(\pi_{\kappa_l}(n),1)$ is a polynomial in $n$ of degree at most $3$, so for $n\ge n_2$ all these determinants have constant signs; so has the relation $\pi_\kappa(n)=\pi_{\kappa'}(n)$, which holds for all $n$ or for at most one.
The hull is $3$-dimensional if and only if one of these determinants is non-zero.
If it is, $S\subseteq\mathcal K$ is the index set of a facet if and only if it contains $\kappa_1,\kappa_2,\kappa_3$ such that some $\Delta_{(\kappa_1,\kappa_2,\kappa_3,\kappa)}(n)\ne0$, $S=\{\kappa:\Delta_{(\kappa_1,\kappa_2,\kappa_3,\kappa)}(n)=0\}$, and the $\Delta_{(\kappa_1,\kappa_2,\kappa_3,\kappa)}(n)$ with $\kappa\notin S$ all have the same sign.
A point is a vertex if and only if its index lies in at least three facet index sets; $f_0$ is the number of vertices up to coincidence, $f_2$ the number of facet index sets, and $f_1=f_0+f_2-2$.
All of this depends only on the signs.
\end{proof}

\begin{proof}[Proof of \Cref{prop:fixed-steps}, (a) and (b)]
Apply \Cref{lem:linear-polygons} to each $P_{ab}(n)$, $(a,b)\in\mathcal L$, let $n_1$ be the largest of the thresholds and $D=\operatorname{lcm}(\delta,D_{ab}:(a,b)\in\mathcal L)$, which divides $\delta L$.
For $r\in\{0,\dots,D-1\}$ let $\mathcal K_r$ be the set of $(a,b,p,q)$ with $(a,b)\in\mathcal L$ and $(p,q)\in Z^{ab}_{r\bmod D_{ab}}$, and $\pi_{(a,b,p,q)}(n)=(np+q,-na-b)$.
By \eqref{eq:levels} and \Cref{lem:linear-polygons}, $\Psi_n(\mathcal G_n)=\operatorname{conv}\{\pi_\kappa(n):\kappa\in\mathcal K_r\}$ for $n\in\mathcal N$, $n\ge n_1$, $n\equiv r\pmod D$.
So $\mathcal G_n$ has at most $\max_r|\mathcal K_r|$ vertices for $n\ge n_1$, and at most $n_1^2$ before; this is (a).
Let $n_0$ be the largest of $n_1$ and the thresholds $n_2$ of \Cref{lem:moving-points} for the families $(\pi_\kappa)_{\kappa\in\mathcal K_r}$.
For $n\in\mathcal N$ with $n\ge n_0$ also $n+D\in\mathcal N$, since $\delta\mid D$, and $n+D\equiv n\pmod D$; so \Cref{lem:moving-points} gives (b) with $D$ in place of $\delta L$, and repeating it $\delta L/D$ times gives (b).
\end{proof}

\paragraph{Rubin's triangle.}
Comparing entries in $\begin{pmatrix}1&1\\1&0\end{pmatrix}^t=\begin{pmatrix}F_{t+1}&F_t\\F_t&F_{t-1}\end{pmatrix}$, $t\in\Z$, gives $F_{-t}=(-1)^{t+1}F_t$, so that $F_t\ge1$ for every odd $t$, with equality only for $t=\pm1$, and d'Ocagne's identity
\begin{equation}
\label{eq:docagne}
F_rF_{s+1}-F_{r+1}F_s=(-1)^sF_{r-s}\qquad(r,s\in\Z).
\end{equation}

\begin{proof}[Proof of \Cref{lem:fib-triangle}]
Write $P_t=(F_{2t},b-F_{2t-1})$.
For $0\le j\le k$ put $\mathbf n_j=(F_{2j},F_{2j+1})$ and $c_j=bF_{2j+1}-1$, so that $\mathbf n_k=(a,b)$ and $c_k=R$, and let $H$ be the polygon defined by $x\ge0$, $y\ge0$ and $\mathbf n_j\cdot(x,y)\le c_j$ for $0\le j\le k$; then $H\subseteq T_k$.
Every lattice point of $T_k$ lies in $H$: the maximum of $\mathbf n_j$ on $T_k$ is attained at a vertex, $(0,0)$, $(R/a,0)$ or $(0,R/b)$, and since $aF_{2j+1}-bF_{2j}=F_{2k-2j}\ge0$ by \eqref{eq:docagne}, it is $F_{2j+1}R/b<bF_{2j+1}$; at a lattice point $\mathbf n_j\cdot(x,y)$ is an integer, hence at most $c_j$.
By \eqref{eq:docagne}, $c_j-\mathbf n_j\cdot P_t=F_{2t-1}F_{2j+1}-F_{2t}F_{2j}-1=F_{2(t-j)-1}-1\ge0$, with equality exactly for $t\in\{j,j+1\}$.
The coordinates of $P_t$ are non-negative, as $F_{2t-1}\le F_{2k+1}=b$ for $0\le t\le k+1$; so $(0,0)$ and the $P_t$ lie in $H$.
Each of these $k+3$ points lies on two non-parallel boundary lines of $H$: $(0,0)$ on $x=0$ and $y=0$, $P_0=(0,b-1)$ on $x=0$ and $\mathbf n_0\cdot(x,y)=c_0$, $P_t$ for $1\le t\le k$ on the lines of $\mathbf n_{t-1}$ and $\mathbf n_t$ (with $\det(\mathbf n_t,\mathbf n_{t-1})=F_2=1$ by \eqref{eq:docagne}), and $P_{k+1}=(F_{2k+2},0)$ on $y=0$ and the line of $\mathbf n_k$.
So they are $k+3$ vertices of $H$, distinct since the $F_{2t}$ with $t\ge1$ are distinct and positive and $b-1\ge1$; as $H$ is defined by $k+3$ inequalities, it has no others.
Hence $H$ is the convex hull of lattice points of $T_k$, and $\operatorname{conv}(T_k\cap\Z^2)=H$.
\end{proof}

\end{document}